\documentclass{article}
\usepackage {amssymb}
\usepackage {amsmath}
\usepackage{wrapfig}
\usepackage {amsthm}
\usepackage{graphicx}
\usepackage{subfigure}
\usepackage {amscd}
\usepackage[colorlinks,linkcolor=black, anchorcolor=black, citecolor=red]{hyperref}
\usepackage{xypic}
\theoremstyle{plain}
 \newtheorem{theorem}{Theorem}%[section]
 \newtheorem{prop}{Proposition}%[section]
 \newtheorem{lem}{Lemma}%[section]
 \newtheorem{cor}{Corollary}%[section]

\theoremstyle{definition}
 \newtheorem{remark}{Remark}%[theorem]
 \newtheorem{definition}{Definition}%[theorem]

\newcommand{\mX}{{\mathcal{X}}}
\newcommand{\mD}{{\mathcal{D}}}
\newcommand{\mU}{{\mathcal{U}}}
\newcommand{\mL}{{\mathcal{L}}}
\newcommand{\mW}{{\mathcal{W}}}
\newcommand{\mV}{{\mathcal{V}}}
\newcommand{\sddbar}{{\sqrt{-1}\partial\bar{\partial}}}

\begin{document}

\title{Yau-Tian-Donaldson Correspondence for K-semistable Fano Manifolds}
\author{Chi Li}
\date{}
\maketitle

\begin{abstract}
In this paper, using the recent compactness results of Tian and Chen-Donaldson-Sun, we prove the K-semistable version of Yau-Tian-Donaldson correspondence for Fano manifolds.
\end{abstract}

\tableofcontents %% Just for papers exceeding 50 pages.

%\section{Introduction}\label{sec:intro}

\section{Introduction}
The recent development in K\"{a}hler geometry is the announcement of resolution of the Yau-Tian-Donaldson's conjecture for Fano manifolds, first by Tian and independently by Chen-Donaldson-Sun (See \cite{Ti12}, \cite{CDS}).  Before stating the results, we recall some terminology. Let $X$ be a smooth Fano manifold. This means that $X$ is a compact complex manifold with an ample anti-canonical line bundle $-K_X$. In other words, the first Chern class $c_1(X)=c_1(-K_X)\in H^{1,1}(X,\mathbb{Z})$ is positive. Let $\omega_0$ be a smooth K\"{a}hler metric in $2\pi c_1(X)$. Define the space of smooth K\"{a}hler potentials
\[
\mathcal{PSH}^{sm}(\omega_0)(X)=\{\phi\in C^{\infty}(X,\mathbb{R}); \omega_\phi:=\omega+\sddbar\phi>0\}.
\]
The Ricci curvature and the scalar curvature of the K\"{a}hler metric $\omega_\phi$ can be calculated as follows:
\[
Ric(\omega_\phi)=-\sddbar\log\omega_\phi^n=:Ric(\omega_\phi^n),\quad S(\omega_\phi)=g_{\phi}^{i\bar{j}}Ric(\omega_\phi)_{i\bar{j}}.
\]
Note that to define the logarithm of any volume form $d\mu$, we implicitly choose a coordinate chart $z=\{z_i\}$ and denote $dz=dz_1\wedge\dots \wedge dz_n$, so that we can write:
\begin{equation}\label{Ricvolume}
Ric(d\mu):=-\sddbar\log d\mu=-\sddbar\log\frac{d\mu}{dz\wedge d\bar{z}}=-\sddbar\log|\partial_z|_{d\mu}^2.
\end{equation}
It's easy to verify that this is independent of coordinate charts. More intrinsically, $d\mu$ defines a Hermitian metric $|\cdot|_{d\mu}^2$ on $-K_X$ and the above $Ric(d\mu)$ is the Chern curvature of this hermitian metric. In particular $Ric(d\mu)$ is a closed $(1,1)$-form representing the cohomology class $2\pi c_1(X)$. As a consequence, the integral of the scalar curvature of $\omega_\phi \in 2\pi c_1(-K_X)$ is a topological constant:
\[
\int_X S(\omega_\phi)\omega_\phi^n=n\int_X Ric(\omega_\phi)\wedge \omega_\phi^{n-1}=n (2\pi)^n\langle c_1(X)^{n}, [X]\rangle=n V, 
\]
where we have denoted $V=(2\pi)^n c_1(X)^n$. The additive group $(\mathbb{R}, +)$ acts on $\mathcal{PSH}^{sm}(\omega_0)$ by addition. By the $\partial\bar{\partial}$-Lemma, the space of smooth K\"{a}hler metrics in $2\pi c_1(X)$ is the same as
\[
\overline{\mathcal{PSH}^{sm}}(\omega_0)=\mathcal{PSH}^{sm}(\omega_0)/\mathbb{\mathbb{R}}.
\]
The (normalized) Ricci potential $h_{\omega_0}$ of $\omega_0$ measures the deviation of $\omega_0$ from being K\"{a}hler-Einstein. It is defined by the identities:
\begin{equation}\label{ricpot}
Ric(\omega_0)-\omega_0=\sddbar h_{\omega_0},\quad \int_X e^{h_{\omega_0}}\omega_0^n=\int_X\omega_0^n.
\end{equation}
$\omega_{\phi}\in 2\pi c_1(X)$ is called K\"{a}hler-Einstein if $Ric(\omega_\phi)=\omega_{\phi}$ (In other words, $h_{\omega_{\phi}}=0$). This condition is equivalent to that $\phi$ 
satisfies a complex Monge-Amp\`{e}re equation:
\begin{equation}\label{KEeq}
Ric(\omega_\phi)=\omega_\phi \quad \Longleftrightarrow\quad  (\omega_0+\sddbar\phi)^n=e^{h_{\omega}-\phi}\omega^n.
\end{equation}
Now we define the special degeneration following Tian \cite{Ti97} (See also \cite{LX11}).
Let $\mathbb{Q}\ni\lambda>0 $ and fix $D\in |-\lambda K_X|$ to be a smooth divisor which is linearly equivalent to a positive multiple of the anti canonical divisor. 
\begin{definition}
Let $0\le\alpha<1$. 
\begin{enumerate}
\item
A special degeneration of $(X, \alpha D)$ is a $\mathbb{C}^*$-equivariant map $\pi: (\mX,\mD)\rightarrow \mathbb{C}$ satisfying
\begin{enumerate}
\item The general fibre $(\mX_t,\mD_t)\cong (X,D)$ for $t\neq 0$.
\item the central fibre $\mX_0=\pi^{-1}\{0\}$ is a $\mathbb{Q}$-Fano variety and $(\mX_0, \alpha\mD_0)$ is a klt pair.
\end{enumerate}
(For the definition of $\mathbb{Q}$-Fano varieties and klt pairs, see the classical reference in birational geometry by Koll\'{a}r-Mori \cite{KM98}.)
\item
Following Ding-Tian \cite{DT92}, we define the generalized log-Futaki invariant of the $(\mX, \alpha\mD, -K_{\mX/\mathbb{C}})$ as the log-Futaki-invariant (\cite{Do11}, see also \cite{Li11}) on the central fibre as follows:
\begin{eqnarray*}
&&Fut(\mX,\alpha\mD, -K_{\mX/\mathbb{C}})=Fut(\mX_0,\alpha\mD_0,v)\\
%\int_X \theta_v Ric(\omega-\omega)\wedge\omega^{n-1}
%&=&-\int_X v(h_{\omega})\omega^n-n\alpha\left(\int_{2\pi D}\theta_v\omega^{n-1}-\frac{\langle c_1(-K_X)^{n-1}, 2\pi D\rangle }{\langle c_1(-K_X)^{n},[X]\rangle}\int_X\theta_v\omega^n\right).
&=&n\int_{\mX_0} \theta_v (Ric(\omega)-\omega)\wedge\omega^{n-1}-2\pi n\alpha\left(\int_{\mD_0}\theta_v\omega^{n-1}-\lambda %\frac{\langle c_1(-K_X)^{n-1}, 2\pi D\rangle }{\langle c_1(-K_X)^{n},[X]\rangle}
\int_{\mX_0}\theta_v\omega^n\right).
\end{eqnarray*}
%Here we denote:
%\[
%\mu(X,D)=\frac{\langle c_1(-K_X)^{n-1}, [D]\rangle }{\langle c_1(-K_X)^n,[X]\rangle}.
%\]
$v$ is the generating holomorphic vector of the $\mathbb{C}^*$-action on the central fibre. $\omega\in 2\pi c_1(-K_{\mX_0})$ is a smooth K\"{a}hler metric. $Ric(\omega)$ is the Ricci curvature of $\omega$. $\theta_v$ is the Hamiltonian function for $v$ defined by $\iota_v\omega=\sqrt{-1}\bar{\partial}\theta_v$.  
\item
$(X,\alpha D, -K_X)$ is log-K-semistable (resp. log-K-polystable) if for any special degeneration of $(X,\alpha D)$, the log-Futaki invariant $Fut(\mX,\alpha\mD,-K_{\mX/\mathbb{C}})\ge 0$ (resp. $\ge 0$ and the equality holds if and only if $(\mX, \mD)$  is a product special degeneration, i.e. $(\mX, \mD)\cong (X\times\mathbb{C}, D\times \mathbb{C})$ with the $\mathbb{C}^*$ action induced by some $\mathbb{C}^*$ action on the pair $(X, D)$). 
\end{enumerate}
\end{definition} 
Then we have the following 
\begin{theorem}[Tian \cite{Ti12}, Chen-Donaldson-Sun \cite{CDS}]\label{TCDS}
 If $(X,-K_X)$ is K-polystable, then $X$ admits a K\"{a}hler-Einstein metric.
\end{theorem}
\begin{remark}
The reverse direction, i.e. K\"{a}hler-Einstein implying K-polystability, was proved by Tian (\cite{Ti97}) when $Aut^0(X)$ is discrete, and recently by Berman (\cite{Berm2}) in general. 
\end{remark}
To study the K\"{a}hler-Einstein equation \eqref{KEeq}, there are now two important continuity methods considered in the subject:
\begin{itemize}
\item (Aubin's continuity method) In this continuity method, we consider:
\begin{equation}\label{aubin}
Ric(\omega_{\phi_t})=t\omega_{\phi_t}+(1-t)\omega_0\Longleftrightarrow (\omega_0+\sddbar\phi_t)^n=e^{h_{\omega_0}-t\phi_t}\omega_0^n.
\end{equation}
Define the supreme value of $t$ for the solvability of the above equation as (see \cite{Ti92}, \cite{Sze})
\begin{equation}\label{R(X)}
R(X)=\sup\{t;\;\exists\; \omega\in 2\pi c_1(X) \mbox{ such that }
Ric(\omega)\ge t\omega \}.
\end{equation}

\item (Conical continuity method)
For any $\lambda\ge 1\in\mathbb{Z}$. Let $D=\{s=0\}\in |-\lambda K_X|$ be any smooth pluri-anticanonical divisor. Let $|\cdot|^2=|\cdot|_{h_0}^2$ be the induced Hermitian metric on $-\lambda K_X$ whose Chern curvature is $\lambda \omega_0$. We consider equations:
\begin{equation}\label{conical}
Ric(\omega_{\psi_t})=t\omega_{\psi_t}+2\pi(1-t)\{D\}/\lambda \Longleftrightarrow  (\omega_0+\sddbar\psi_t)^n=e^{h_{\omega_0}-t\psi}\frac{\omega_0^n}{|s|^{2(1-t)/\lambda}}.
\end{equation}
Note that the (strong) solution of this equation corresponds to a conical K\"{a}hler-Einstein metric on $(X,(1-t)D/\lambda)$ which means a K\"{a}hler-Einstein metric with cone singularities along the smooth divisor $D$ of cone angle $2\pi\beta=2\pi(1-(1-t)/\lambda)$ . Similarly as above, we define 
\[
R(X,D/\lambda)=\sup\{t; \; \exists\; \mbox{strong conical K\"{a}hler-Einstein metric on } (X,(1-t)D/\lambda)\}.
\]
Here ``strong" means the solution belongs to the space of $C^{2,\alpha,\beta}$-conical metrics introduced by Donaldson \cite{Do11}. 
\end{itemize}
The program to prove theorem \ref{TCDS} using Aubin's continuity method (\cite{Au78}) was first proposed by Tian in the early 90's following his solution of K\"{a}hler-Einstein problem for del Pezzo surfaces \cite{Ti90}. The hard core of this program is Tian's conjecture of the so called partial $C^0$-estimate (cf. \cite{Ti90}, \cite{Ti09}, \cite{Li12}). The foundational work of Cheeger-Colding-Tian \cite{CCT} is a major step towards this conjecture. Donaldson's great insight \cite{Do11}, which leads to the breakthrough, is that the conical continuity method is more adapted to the problem. People in the field then extend much of the PDE theory in the old continuity method to the conical continuity method (See in particular, \cite{Do11}, \cite{Berm} \cite{JMRL}). However,  to complete the program, one needs to resort to Tian's idea of proving partial $C^0$-estimates and extending Cheeger-Colding-Tian's theory to establish the following important compactness theorem.

\begin{theorem}[Tian \cite{Ti12}, Chen-Donaldson-Sun \cite{CDS}]\label{specialdeg}
Let $\gamma=R(X,D/\lambda)$. As $t\rightarrow \gamma$, the conical K\"{a}hler-Einstein metric $\hat{\omega}_t$ on $(X, (1-t)D/\lambda)$ Gromov-Hausdorff converges to a conical K\"{a}hler-Einstein metric $\hat{\omega}_\gamma$ on a klt pair $(\mX_0, (1-\gamma)\mD_0/\lambda)$. Moreover, there is a special degeneration $(\mathcal{X}, (1-\gamma)\mathcal{D}/\lambda, -K_{\mathcal{X}/\mathbb{C}})$ of $(X, (1-\gamma)D/\lambda, -K_X)$ with $(\mX_0, (1-\gamma)\mD_0/\lambda, -K_{\mX_0})$ being the central fibre. 
\end{theorem}

The purpose of this note is to show the following semistable version of Yau-Tian-Donaldson correspondence using the above compactness result as a tool of blackbox. For the definition of 
Ding energy and Mabuchi energy, see the next section.
\begin{theorem}\label{main}
The following conditions are equivalent:
\begin{enumerate}
\item \label{Ksemist}
$(X,-K_X)$ is K-semistable.
\item \label{R=1}
$R(X)=1$. 
\item \label{B=1}
$R(X,D/\lambda)=1$. 
\item \label{calabi} The infimum of Calabi functional is zero, that is
\[
\inf_{\omega_\phi\in 2\pi c_1(X)}\|S(\omega_\phi)-n\|_{L^2}=0.
\]
\item \label{bddKenergy}
The Ding-energy is bounded from below, or equivalently,  the Mabuchi-energy is bounded from below. 
\end{enumerate}
\end{theorem}
\begin{remark}
Professor Robert Berman pointed out to me that, using Theorem \ref{main}, results in \cite{TZZZ} and  his paper \cite{Berm2}, conditions above are also equivalent to the condition that the supremum of (normalized) Perelman's $\lambda$-functional is equal to $n\cdot {\rm Vol}(X)$. Using the terminology in \cite{Band} and \cite{TiWa}, we could say that K-semistable Fano manifolds are the same as almost K\"{a}hler-Einstein Fano manifolds.
\end{remark}
Many implications of the above conditions were known. See discussions in the next section. Our main contribution is to complete the loop of implications by showing the implication $\ref{Ksemist}\Rightarrow \ref{bddKenergy}$. %Many ingredients for proving the full statement are available after the works of Tian and Chen-Donaldson-Sun. For example, the works in \cite{LiSu}, \cite{SoWa} are related to this result. 
In addition to the compactness result in Theorem \ref{specialdeg}, the main ingredient to proving this is the following result. 
\begin{theorem}\label{bdDing}
Assume $D\in |-\lambda K_X|$ for $\mathbb{Z}\ni\lambda\ge 1$, and we have a special degeneration $(\mX, \alpha\mD)$ of the klt pair $(X, \alpha D)$ such that there is a weak conical K\"{a}hler-Einstein metric on the klt pair $(\mX_0, \alpha\mD_0)$. Then for the pair $(X,\alpha D)$ and any reference K\"{a}hler metric $\omega$, the Ding-energy $F_{\omega,\alpha D}$ is bounded from below.
As a consequence, $(X, \alpha\mD)$ is log-K-semistable.
\end{theorem}
%\begin{theorem}\label{degenergy}
%Suppose there is a special degeneration $(\mX, \mD, -K_{\mX/\mathbb{C}})$ of $(X, D)$ such that $(\mX_0, (1-\gamma)\mD_0/\lambda)$ admits a weak conical K\"{a}hler-Einstein metric, then the log-%Ding-energy $F_{\omega,(1-\gamma)D/\lambda}$, or equivalently the log-Mabuchi-energy $\nu_{\omega,(1-\gamma)D/\lambda}$ on $(X,D)$, is bounded from below.
%\end{theorem}
This is a generalization of Chen's theorem \cite{Chen} from the smooth setting to the general singular setting in the (logarithmic) Fano case. A simple special case of the Theorem \ref{bdDing} already played an important role in our previous work in \cite{LiSu}. Here we prove the above general result by doing explicit calculations in Lemma \ref{contlem} ( or equivalently Lemma \ref{IIcont})
to resolve a technical difficulty in \cite{LiSu}. As an immediate corollary of Theorem \ref{specialdeg} of Tian/Chen-Donaldson-Sun and Theorem \ref{bdDing}, we get:
\begin{cor}
For any smooth integral pluri-anticanonical divisor $D\in |-\lambda K_X|$ for some $\mathbb{Z}\ni \lambda\ge 1$,  the pair $
\left(X, \frac{1-R(X,D/\lambda)}{\lambda}D\right)$ is log-K-semistable. Moreover, $(X, (1-\gamma)D/\lambda)$ is log-K-stable if and only if $\gamma\in (1-\lambda, R(X,D/\lambda))$.
\end{cor}
As mentioned above, in proving Theorem \ref{bdDing}, a technical step relies on the following Lemma which is of independent interest. Denote $B_1(0)=\{z\in\mathbb{C}, |z|<1\}$.
\begin{lem}\label{contlem}
Let $\pi: \mX\rightarrow B_1(0)$ be a family of Fano varieties over the unit-disc such that the general fiber is a smooth Fano manifold and the central fiber is a Fano variety with log terminal singularities. Let $h$ be a continuous metric on the relative anti-canonical bundle $-K_{\mX/B_1(0)}$. Then the function 
\[
{\bf f}(t):=\int_{\mX_t} dV(h),
\]
is continuous as $t\rightarrow 0$.
\end{lem}
See Definition \ref{advol} for the definition of $dV(h)$.
As pointed out to me by the referee, this lemma can be seen as a strengthening of one result by Gross in \cite[Appendix B]{RoZh}. See Remark \ref{DKG} for more discussions. 

%\begin{theorem}[{}\cite{RoZh}{}]
%Let $\pi: \mX\rightarrowB_1(0)$ be a flat projective family of $n$-dimensional Calabi-Yau varieties, with $\mX_t=\pi^{-1}(t)$ non-singular for $t\neq 0$ and $\mX_0=\pi^{-1}(0)$.
%\end{theorem}
The organization of the paper is as follows. In the next section, we briefly recall some preliminary results. In section \ref{sectionthm}, we prove Theorem \ref{main} and Theorem \ref{bdDing} modulo the technical Lemma \ref{contlem}, or equivalently Lemma \ref{IIcont}. In section \ref{sectionlemma}, we prove Lemma \ref{contlem}. In the last section, 
we give examples of log-semistable pairs. 

\section{Preliminary results}

Using notations in the introduction, we recall some known results.  Firstly, to study the relation between $R(X)$ and $R(X,D/\lambda)$, we consider the following functionals. See \cite{Ti99}. (Also see for \cite{Berm, JMRL, Li12} for general twisted functionals.)
\begin{definition}\label{volume}
Let $V=\int_X\omega^n=(2\pi)^n\langle c_1(X)^n, [X]\rangle$. For any $\phi\in\mathcal{PSH}^{sm}(\omega_0)$, we define
\begin{enumerate}
\item (Monge-Amp\`{e}re energy)
\begin{equation}\label{MAenergy}
F_{\omega_0}^0(\omega_\phi)=-\frac{1}{n+1}\frac{1}{V}\sum_{i=0}^{n}\int_X\phi\omega_0^i\wedge\omega_\phi^{n-i}.
\end{equation}
\item (Norm energy)
\[
I_{\omega_0}(\omega_\phi)=\frac{1}{V}\int_X\phi(\omega^n-\omega_\phi^n),\quad J_{\omega_0}(\omega_\phi)=F^0_{\omega_0}(\phi)+\frac{1}{V}\int_X\phi\omega^n.
\]
\item (Ding energy, \cite{Ding, DT93}) 
\begin{enumerate}
\item (For Aubin's continuity method) For $t\neq 0$, define
\[
F_{\omega_0,(1-t)\omega_0}(\omega_\phi)=F_{\omega_0}^0(\phi)-\frac{1}{t}\log\left(\frac{1}{V}\int_X e^{h_{\omega_0}-t\phi}\omega_0^n\right).
\]
%For $t=0$, 
%\[
%F_{\omega_0,\omega_0}(\omega_\phi)=F_{\omega_0}^0(\phi)+\int_X\phi e^{h_{\omega_0}}\omega_0^n.
%\]
\item (log-Ding-energy) For $t\neq 0$, define
\begin{equation*}
F_{\omega_0,(1-t)D/\lambda}(\omega_\phi)=F_{\omega_0}^0(\phi)-\frac{1}{t}\log\left(\frac{1}{V}\int_X e^{h_{\omega_0}-t\phi}\frac{\omega_0^n}{|s|^{2(1-t)/\lambda}}\right).
\end{equation*}
For $\lambda\ge 2$ and $t=0$, we normalize $|\cdot|^2$ such that $\int_X e^{h_{\omega_0}}\omega_0^n/|s|^{2/\lambda}=\int_X\omega_0^n$, and
define 
\[
F_{\omega_0, D/\lambda}(\omega_\phi)=F_{\omega_0}^0(\phi)+\frac{1}{V}\int_X (\phi-\log|s|^{2/\lambda}) e^{h_{\omega_0}}\frac{\omega_0^n}{|s|^{2/\lambda}}.
\]
\end{enumerate}
\item (Mabuchi energy, \cite{Ma86}) \begin{enumerate}
\item (The 2nd formula for the Mabuchi energy appeared in \cite[Proposition 3.1]{Ti94})
\begin{eqnarray*}
\nu_{\omega_0}(\omega_\phi)&=&-\int_0^1\int_X S(\omega_\phi)-n)\dot{\phi}_t\omega_{\phi_t}^n dt\\
&=&\!\frac{1}{V}\int_X\log\frac{\omega_\phi^n}{\omega_0^n}\omega_\phi^n\!+\!\left(\frac{1}{V}\int_X\phi\omega_\phi^n\!+\!F_{\omega_0}^0(\phi)\right)\!+\!\frac{1}{V}\int_Xh_{\omega_0}(\omega_0^n-\omega_\phi^n).
\end{eqnarray*}
\item For Aubin's continuity method, define:
\[
\nu_{\omega_0,(1-t)\omega_0}(\omega_\phi)=\nu_{\omega_0}(\omega_\phi)+(1-t)(I_{\omega_0}-J_{\omega_0})(\omega_\phi).
\]
\item (log-Mabuchi-energy) For the conical continuity method, define:
\begin{equation*}
\nu_{\omega_0,(1-t)D/\lambda}(\omega_\phi)=
\nu_{\omega_0, (1-t)\omega_0}(\omega_\phi)+(1-t)\frac{1}{V}\int_X\log|s|_{h_0}^{2/\lambda}(\omega_\phi^n-\omega_0^n)
%=\nu_{\omega_0}(\omega_\psi)+(1-\beta)(I_{\omega_0}-J_{\omega_0})(\omega_\phi)+(1-\beta)\int_X\log|s|_{h_0}^{2/\lambda}(\omega_\psi^n-\omega_0^n).
\end{equation*}
\end{enumerate}
\end{enumerate}
\end{definition}

Recall the following definition by Tian:
\begin{definition}[\cite{Ti97}]
A functional $F$ on $\overline{\mathcal{PSH}^{sm}}(\omega_0)$ is called linearly proper if there exist constants $C_1=C_1(X)>0$ and $C_2=C_2(X)$, such that for any $\phi\in
\overline{\mathcal{PSH}^{sm}}(\omega_0)$, we have
\[
F(\omega_\phi)\ge C_1 I_{\omega_0}(\omega_\phi)-C_2. 
\]
\end{definition}
This condition is equivalent to a strong Moser-Trudinger-Onofri inequality \cite{DT93}. The following proposition summarizes the relevant PDE theory from the variational point of view for both Aubin's continuity method and the conical continuity method. 
\begin{prop}\label{PDEcont}
\begin{enumerate}
\item %(\cite{Ding}, \cite{Ti97}, \cite{Sze}, \cite{Berm}, \cite{Li12}, \cite{SoWa})
$R(X)\ge r_0$ if and only if $F_{\omega_0, (1-t)\omega_0}$ is linearly proper when $t<r_0$, equivalently, if and only if $\nu_{\omega_0,(1-t)\omega_0}$ is linearly proper when $t<r_0$.
\item %(\cite{JMRL}, \cite{Berm}, \cite{Li12}, \cite{LiSu}, \cite{SoWa})
$R(X,D/\lambda)\ge b_0$ if and only if $F_{\omega_0, (1-t)D/\lambda}$ is linearly proper when $t<b_0$, equivalently, if and only if $\nu_{\omega_0, (1-t)D/\lambda}$ is linearly proper when $t<b_0$. 
\end{enumerate}
\end{prop}
%\begin{remark}
The proposition is now well known (cf.  \cite{Berm}, \cite{LiSu}, \cite{SoWa}, \cite{Sze}). But for the reader's convenience, I will sketch the proof of this proposition. See \cite[Chapter 3]{Li12} for more discussions and references related to this proposition.
%\end{remark}
\begin{proof}
We first assume the solvability of \eqref{aubin} (resp.  \eqref{conical}). Then the argument splits into the following steps.
\begin{enumerate}
\item
When $0<t_1\ll 1$, the functionals $F_{\omega_0,(1-t_1)\omega_0}$, $\nu_{\omega_0,(1-t_1)\omega_0}$ (resp. $F_{\omega_0,(1-t_1)D/\lambda}$, $\nu_{\omega_0,(1-t_1)D/\lambda}$ ) are proper. This can be proved for Mabuchi energy using Tian's $\alpha$-invariant (\cite{Ding}, \cite{Ti99}) (resp. log-$\alpha$-invariant \cite{Berm, JMRL}). Then by \cite{Berm} and \cite{Rub}, the properness of Mabuchi energy and properness of Ding energy are equivalent.
\item
The solution of the equation \eqref{aubin} (resp. \eqref{conical}) for $t=t_2<R(X)$ (resp. $t=t_2<R(X,D/\lambda)$) obtains the minimum of both functionals $F_{\omega_0,(1-t_2)\omega_0}$ and $\nu_{\omega_0,(1-t_2)\omega_0}$  (resp. $F_{\omega_0, (1-t_2)D/\lambda}$ and $\nu_{\omega_0,(1-t_2)D/\lambda}$). (cf. \cite{BM87}, \cite{Bern}, \cite{Berm}, \cite{BBGZ})
\item
The functional $\nu_{\omega_0,(1-t)\omega_0}$ (resp. $\nu_{\omega_0, (1-t)D/\lambda}$) is linear in the variable $t$.  Also, by H\"{o}lder's inequality, the functional $F_{\omega_0, (1-t)\omega_0}$ (resp. $F_{\omega_0, (1-t)D/\lambda}$) is concave downward in the variable $t$. (See \cite{LiSu}) So the properness of functionals for intermediate values $0<t_1<t<t_2<R(X)$ (resp. $0<t_1<t<t_2<R(X,D/\lambda)$) follows from interpolations.
\end{enumerate} 
To prove the other direction, we assume the properness of functionals. 
\begin{enumerate}
\item
Start the continuity method. For Aubin's continuity method \eqref{aubin}, one can get the solution at $t=0$ using Yau's theorem \cite{Yau} on prescribing Ricci curvatures. In fact, Tian's $\alpha$-invariant \cite{Ti87} bounds $R(X)$ away from 0 by the inequality $R(X)\ge\frac{n+1}{n}\alpha(X)>0$. For the conical continuity method \eqref{conical}, for simplicity, we assume $\lambda\ge 2$. One can choose $t_0$ such that $1-\lambda< t_0< 1-\lambda+\epsilon<0$ with $0<\epsilon\ll 1$.  The solvability of \eqref{conical} at $t=t_0$ is proved in the same way as in the case of smooth K\"{a}hler-Einstein metric with negative Ricci curvature proved by Aubin and by Yau \cite{Yau}. Because the cone angle $0< 2\pi\beta=2\pi(1-(1-t_0)/\lambda)\ll 1$ is very small, the cone singularities do not cause troubles in the analysis. In general, all cases (for all $\lambda\ge 1$) can be dealt with as in \cite{Berm}, \cite{JMRL}. See also \cite{LiSu}.
\item
The openness of solution set holds for \eqref{aubin} by \cite{Au78} (see also \cite{Ti99}) and for \eqref{conical} by \cite{Do11}.
\item
The closedness of solution set follows from a priori  estimates which are true by the assumption that the functionals are proper. This is proved as follows. Firstly, by the interpolation argument as above, the functionals are uniformly proper along continuity methods. So $I_{\omega_0}(\omega_\phi)$ are uniformly bounded. Secondly, one can show that the Sobolev constants are uniformly bounded along the continuity methods, and so the $C^0$-estimates follow from Moser iterations. (cf. \cite{Ti87}, \cite{JMRL}) Then by the theory of (singular) Monge-Amp\`{e}re equations, $C^0$-estimates are sufficient for higher order estimates. For more details, see \cite{Au78}, \cite{JMRL}.
\end{enumerate}

\end{proof}
\begin{cor}\label{corollary1}
\begin{enumerate}
\item 
$R(X,D/\lambda)\le R(X)$.
\item (\cite{SoWa})
$R(X,D/\lambda)\ge R(X) (\lambda-1)/(\lambda-R(X))$.
\end{enumerate}
\end{cor}
\begin{proof}
\begin{enumerate}
\item
Because $F_{\omega_0, (1-t)\omega_0}\ge F_{\omega_0, (1-t)D/\lambda}-C$ for constant $C=\frac{1-t}{\lambda}\log \max_X|s|^2$, this follows from the above Proposition immediately.
\item
By the H\"{o}lder's inequality, we have
\[
\int_X e^{h_{\omega_0}-t\psi}\frac{\omega_0^n}{|s|^{2(1-t)}}\le e^{\|h_{\omega_0}\|_{L^{\infty}}/q} \left(\int_X e^{h_{\omega_0}-tp\psi}\omega_0^n\right)^{1/p}\left(\int_X\frac{1}{|s|^{2(1-t)q/\lambda}}\omega_0^n\right)^{1/q}.
\]
Then the inequality follows by solving the following conditions:
\[
t\cdot p< R(X), \quad p^{-1}+q^{-1}=1,\quad (1-t)q/\lambda<1.
\]
\end{enumerate}
\end{proof}
\begin{remark}
If we define the invariant
\[
R_c(X)=\sup\{ R(X,D/\lambda); \mbox{ smooth } D\in |-\lambda K_X|, \lambda\in\mathbb{Z}_{\ge 1}\},
\]
then the above corollary immediately implies $R_c(X)=R(X)$. However, $R(X,D/\lambda)$ may vary with the divisor $D$ when $R(X)<1$. In \cite{LiSu}, on toric Fano manifolds, we constructed some $D$ and $\lambda$ with $R(X,D/\lambda)=R(X)$. But,  by Sz\'{e}kelyhidi \cite{Sze2}, there exist smooth divisors $D'$ with  $R(X,D'/\lambda)<R(X)$.
\end{remark}
As mentioned before,  many implications of different conditions in Theorem \ref{main}  are already well known. We summarize them in the following items.
\begin{itemize}
\item
$\ref{bddKenergy}\Rightarrow \ref{Ksemist}$. This direction is well known. See Tian \cite{Ti97}, Paul-Tian \cite{PT09}. (See also discussions in \cite{Li12})
\item
$\ref{bddKenergy}\Rightarrow \ref{R=1}$. This was proved by Bando \cite{Band}.
\item
$\ref{bddKenergy}\Rightarrow\ref{B=1}$. This follows from the interpolation argument used in Li-Sun \cite{LiSu}, Song-Wang \cite{SoWa}. 
\item
$\ref{bddKenergy}\Rightarrow\ref{calabi}$. This follows from the work by Bando \cite{Band}.
\item
$\ref{calabi}\Rightarrow\ref{Ksemist}$. This was proved by Donaldson \cite{Do05}.
\item
$\ref{B=1}\Rightarrow\ref{R=1}$. This follows because in general we have $R(X,D/\lambda)\le R(X)$ by Corollary \ref{corollary1}.(i). 
\item
$\ref{R=1}\Rightarrow \ref{Ksemist}$. This was proved by Sz\'{e}kelyhidi \cite{Sze}.
\item
$\ref{R=1}\Rightarrow\ref{B=1}$ when $\lambda\ge 2$. This follows from Song-Wang's estimate in Corollary \ref{corollary1} .(ii).
\item
$\ref{Ksemist}\Rightarrow \ref{B=1}$. This was actually showed by Chen-Donaldson-Sun \cite{CDS} and Tian \cite{Ti12}. See Proposition \ref{KsemitoB} for an explanation.
\end{itemize}
Note that the above implications say that the condition \ref{Ksemist} of K-semistability  is the weakest, while the condition \ref{bddKenergy} that Mabuchi-energy is bounded from below  is the strongest. In order to complete the proof of Theorem \ref{main}, we only need to show that $\ref{Ksemist} \Rightarrow \ref{B=1}$ and $\ref{B=1}\Rightarrow \ref{bddKenergy}$ hold, so that \ref{Ksemist} is indeed equivalent to \ref{bddKenergy}. The first implication $\ref{Ksemist} \Rightarrow \ref{B=1}$ was essentially shown by Tian \cite{Ti12} and Chen-Donaldson-Sun \cite{CDS}. For completeness we outline how to deduce this implication from the compactness result of Theorem \ref{specialdeg}.
\begin{prop}\label{KsemitoB}
$\ref{Ksemist} \Rightarrow  \ref{B=1}$.
\end{prop}
\begin{proof}
Prove by contradiction. Suppose $\gamma=R(X,D/\lambda)<1$.  Then by Theorem \ref{specialdeg}, we get a special degeneration $(\mX, (1-\gamma)\mD)$ of $(X, (1-\gamma)D)$ such that $(\mX_0,(1-\gamma)\mD_0/\lambda)$ admits a conical K\"{a}hler-Einstein metric. If the total space $\mX_0$ is smooth, then by \cite{Do11} or \cite{Li11}, we have 
\[
Fut(\mX_0,(1-\gamma)\mD_0/\lambda,v)=Fut(\mX,(1-\gamma)\mD/\lambda, -K_{\mX/\mathbb{C}})=0,
\] 
where $v$ is the holomorphic vector field on the central fibre $(\mX_0,\mD_0)$ coming from the $\mathbb{C}^*$-action. If $\mX_0$ is singular, there are various ways to get this. One way is to use two 
known results. First, by \cite{BBEGZ} the conical K\"{a}hler-Einstein metric obtains the minimum of log-Mabuchi-energy. Second, the log-Futaki-invariant is the derivative of the log-Mabuchi-energy along a path in the space of K\"{a}hler metrics generated by a one-parameter subgroup (\cite{Li11}). Note that the singularities do not cause troubles in the calculations by the work of \cite{BBEGZ}.  One can also use the log-Ding-energy instead of the log-Mabuchi-energy. To the author's knowledge, this idea was first used by Tian \cite[page 73]{Ti99} in the smooth absolute case and by Berman \cite{Berm2} in the singular setting, and also used by \cite{CDS}. Note that in the singular setting, one can actually lift all the calculations to a log resolution similar to the calculations in the author's thesis \cite{Li12}.

On the other hand, by \cite{LiSu} when $0<\epsilon\ll 1$, $F_{\omega_0, (1-\epsilon)D/\lambda}$ (or $\nu_{\omega_0, (1-\epsilon)D/\lambda}$) is proper and so (see \cite{LiSu})
\[
Fut(\mX_0, (1-\epsilon)\mD_0/\lambda, v)=Fut(\mX, (1-\epsilon)\mD/\lambda, -K_{\mX/\mathbb{C}})>0.
\]
 By the linearity of log-Futaki invariant $Fut(\mX_0, (1-t)\mD_0/\lambda, v)$ in the variable $t$, we get 
$Fut(\mX, (1-1)\mD/\lambda, -K_{\mX/\mathbb{C}})=Fut(\mX, -K_{\mX/\mathbb{C}})<0$.  This is in contradiction with $(X, -K_X)$ being K-semistable. 
\end{proof}

\section{Proof of Theorem \ref{main} and Theorem \ref{bdDing}}\label{sectionthm}

By the above discussions, we just need to prove the following proposition.
\begin{prop}\label{BtoK}
$\ref{B=1}\Rightarrow \ref{bddKenergy}$.
\end{prop}
Given Theorem \ref{specialdeg}, Proposition \ref{BtoK} follows from Theorem \ref{bdDing} which was conjectured in \cite{LiSu}, and the fact that the lower boundedness of Ding energy is equivalent to  the lower boundedness of Mabuchi energy. The latter was first proved independently by H. Li \cite{LH} and Y. Rubinstein \cite{Rub}. Also see  \cite{Berm} and the references therein.
%\begin{prop}\label{bdDing}
%Assume $D\in |-\lambda K_X|$ for $\mathbb{Z}\ni\lambda\ge 2$, and we have a special degeneration $(\mX, \alpha\mD)$ of $(X, \alpha D)$ such that there is a conical K\"{a}hler-Einstein metric on the klt pair $(\mX_0, \alpha\mD_0)$. Then the Ding-energy $F_{\omega_0,\alpha D}$ is bounded from below.
%\end{prop}
Now we will prove Theorem \ref{bdDing} by following the line of arguments in \cite{LiSu} where the result was proved in a simple case of one isolated singularity. Here some more arguments are needed. See Lemma \ref{IIcont}.  Before starting the proof of Theorem \ref{bdDing},  we need some more definitions and remarks.
 \begin{definition}\label{advol}
 Assume $(X,\alpha D)$ is a klt pair, where $D=\{s=0\}\in |-\lambda K_X|$.  Assume that $mK_X$ is Cartier for $\mathbb{Z}\ni m\ge 1$. For any small open set $U\subset X$, let $v$ be a local generator of $\mathcal{O}(mK_X)(U)$ and $v^*$ be the dual generator of $\mathcal{O}(-mK_X)(U)$. For any Hermitian metric $h$ on $-K_X$ with bounded potentials, we define the \textbf{adapted volume form} by the formula:
 \[
 dV((X,\alpha D), h)=(\sqrt{-1})^n\frac{|v^*|_{h}^{2/m}(v\wedge\bar{v})^{1/m}}{|s|_{h}^{2\alpha}}=\frac{dV(h)}{|s|_h^{2\alpha}}.
 \] 
In the following we will just write $dV(h)$ for $dV((X,0),h)$. 
\end{definition}
 \begin{remark}
On a smooth manifold $X$, assume $z=\{z_i\}$ is a local coordinate chart. For any smooth Hermitian metric $h$ on $-K_X$ with Chern curvature equal to $\omega=-\sddbar\log h$, the adapted volume form is given by 
\[
dV(h)=(\sqrt{-1})^n |\partial_z|_h^2 dz\wedge d\bar{z}. 
\]
and it satisfies the equation:
\begin{equation}\label{transadapt}
e^{h_{\omega}}\omega^n=\frac{dV(h)}{V^{-1}\int_XdV(h)}.
\end{equation}
So $ -\sddbar\log (e^{h_{\omega}}\omega^n)=\omega=-\sddbar\log dV(h)$ (See \eqref{Ricvolume}). 
In the previous Monge-Amp\`{e}re equations and definition of  (log-)Ding energies, we could use either $e^{h_{\omega}}\omega^n$ or adapted volume form $dV(h)$. But in the following to apply Berndtsson's subharmonicity result, we need to work with adapted volume forms.
\end{remark} 

\begin{remark}
Since the klt property is important for us in the following, we briefly explain why the klt property holds in Theorem \ref{specialdeg}. For more details, see \cite{BBEGZ}, \cite{CDS}, \cite{Ti12}. 

In general log-setting, $(\mX_0, (1-\gamma)\mD_0/\lambda)$ being Kawamata log terminal is equivalent to the integrability of 
$dV(h)/|s_0|_{h}^{2(1-\gamma)/\lambda}$ where $s_0\in |-\lambda K_{\mX_0}|$ satisfies $\mathcal{D}_0=\{s_0=0\}$, and $h$ is any Hermitian metric on $-K_{\mX_0}$ with bounded local potentials. If $\hat{\omega}_\gamma$ is a conical K\"{a}hler-Einstein metric  on $(\mX_0, (1-\gamma)\mD_0/\lambda)$ and $h=\hat{h}_\gamma$ is the corresponding Hermitian metric on $-K_{\mX_0}$,  then it's easy to verify that, on any small open set $U_0\subset\mX_0$,
\[
dV(\hat{h}_\gamma)=(\sqrt{-1})^n|v_0^*|_{\hat{h}_{\gamma}}^{2/m}(v_0\wedge\bar{v}_0)^{1/m}=e^{f} |s_0|_{\hat{h}_\gamma}^{2(1-\gamma)/\lambda}\hat{\omega}_\gamma^n, 
\]
where $f$ is a bounded pluriharmonic function on $U_0$. So we have
\[
\frac{dV(\hat{h}_\gamma)}{|s_0|_{\hat{h}_\gamma}^{2(1-\gamma)/\lambda}}=\frac{(\sqrt{-1})^n|v_0^*|_{\hat{h}_\gamma}^{2/m}(v_0\wedge\bar{v}_0)^{1/m}}{|s_0|_{\hat{h}_\gamma}^{2(1-\gamma)/\lambda}\hat{\omega}_\gamma^n}\hat{\omega}_\gamma^n=e^{f}\hat{\omega}_\gamma^n.
\]
So, in particular, the klt property of $(\mX_0, (1-\gamma)\mD_0)$ in Theorem \ref{specialdeg} follows from the following two facts.
\begin{itemize}
\item
As remarked before, proving Tian's partial $C^0$-estimates is an important step to prove Theorem \ref{specialdeg}.  By its proof, we know that there exists a non-vanishing generator $v_0^*$ of $\mathcal{O}(-m K_{\mX_0})(U_0)$ for some positive integer $m$.
\item
The volume form $\int_{U_0}\hat{\omega}_{\gamma}^n=Vol_{\hat{\omega}_{\gamma}}(U_0)$ is finite, because $\hat{\omega}_{\gamma}$ is the Gromov-Hausdorff limit of the conical K\"{a}hler-Einstein metrics $\hat{\omega}_{t}$ on $(X, (1-t)D/\lambda)$, as $t\rightarrow\gamma$.
\end{itemize}
\end{remark}
\begin{proof}[Proof of Theorem \ref{bdDing}]
For simplicity of notations, we will mainly concentrate on the absolute case, i.e. when $D=\mD=\emptyset$. We will point out the straightforward modification of arguments for the log setting in Remark \ref{logmod}. 
\begin{enumerate}
\item[{\bf Step 1:}]
We can first embed the special degeneration $\pi: \mX\rightarrow\mathbb{C}$ equivariantly into $\pi_2: \mathbb{P}^N\times \mathbb{C}\rightarrow\mathbb{C}$ using the complete linear system $|-m K_{\mX/\mathbb{C}}|$ and then restrict  $\frac{1}{m}\omega_{FS}^{\mathbb{P}^N}+a\cdot dw\wedge d\bar{w}$ to $\mX$ in order to get a reference K\"{a}hler metric $\Omega$ on $\mX$. Without loss of generality, we can take $\omega_0=\Omega|_{\mX_1}=\Omega|_{X}$. 
\item[{\bf Step 2:}]
For any smooth function $\phi\in \mathcal{PSH}^{sm}(\omega_0)$, we want to construct a geodesic ray from $\omega_\phi$ in the space of K\"{a}hler metrics using the special degeneration $\mX$. It's
well known that this is equivalent to solving the following homogeneous complex Monge-Amp\`{e}re equation on $\mX|_{B_1(0)}=\pi^{-1}(B_1(0))$:
\begin{equation}\label{HCMAbefore}
\Omega+\sqrt{-1}\partial\bar{\partial}\Phi\ge 0, \; (\Omega+\sddbar\Phi)^n=0 \mbox{ on } \mX|_{B_1(0)}; \; \Phi|_{\partial (\mX|_{B_1(0)})=X\times S^1}=\phi.
\end{equation}
It's not difficult to get a bounded solution $\Phi\in L^{\infty}(\mX|_{B_1(0)})$. This is because that we can construct a subsolution as a barrier and Perron's method will give us a bounded solution.  %by using a cut-off function to glue the metric 
To construct such a subsolution, note that on $\mX^*:=\mX\backslash \mX_0\cong X\times\mathbb{C}^*$, we can write: 
\[
\omega_\phi+a\cdot dw\wedge d\bar{w}=\omega_0+\sddbar (\phi+a(|w|^2-1))=\Omega+\sddbar\Psi.
\]
Choose a radially symmetric cut-off function $\eta: B_1(0)\rightarrow \mathbb{R}$ such that $\eta(|w|)=0$ when $|w|<1/3$ and $\eta(|w|)=1$ when $|w|>2/3$. We then define $
\underline{\Psi}=\eta(|w|)\Psi$.  
It's easy to verify that when $a$ is sufficiently large $\underline{\Psi}$ satisfies $\Omega+\sddbar\underline{\Psi}\ge 0$, and hence $\underline{\Psi}$ is a sub solution to the equation \eqref{HCMAbefore}.
For more details, see \cite{PS09} and \cite{LiSu}.

For later estimates, we need the uniform continuity of $\Phi$ away from $\mX^{sing}$. Fortunately, Phong-Sturm \cite{PS09} has proved that $\Phi$ is even uniformly $C^{1,\alpha}$ away from $\mX^{sing}$. For the reader's convenience, we explain briefly how this was achieved in \cite{PS09}. To get higher regularity, Phong-Sturm lifted the problem from the singular space $\mX$ to a smooth space by taking a resolution of singularities $\mu: \tilde{\mX}\rightarrow \mX$. Then the equation \eqref{HCMAbefore} is lifted to the following Dirichlet problem of homogeneous Monge-Amp\`{e}re equation on $\tilde{\mX}|_{B_1(0)}=(\pi\circ \mu)^{-1}(B_1(0))$:  
\[
\mu^*\Omega+\sqrt{-1}\partial\bar{\partial}\tilde{\Phi}\ge 0, \; (\mu^*\Omega+\sqrt{-1}\partial\bar{\partial}\tilde{\Phi})^{n+1}=0 \mbox{ on } \tilde{\mX}|_{B_1(0)} ; \; \tilde{\Phi}|_{\partial (\tilde{\mX}|_{B_1(0)})=S^1\times X}=\phi.
\]
Then they approximated this degenerate equation by a family of non-degenerate  complex Monge-Amp\`{e}re equations:
\[
(\Omega_\epsilon+\sddbar \tilde{\Phi}_{\epsilon})^{n+1}=\epsilon \Omega_\epsilon^{n+1} \mbox{ on } \tilde{\mX}|_{B_1(0)};\; \tilde{\Phi}_\epsilon |_{\partial (\tilde{\mX}|_{B_1(0)})}=\phi.
\]
Here $\Omega_\epsilon=\mu^*\Omega+\epsilon\sddbar\log \|\cdot\|_{E}^2$, where $\|\cdot\|_{E}$ is a Hermitian metric on the line bundle $\mathcal{O}_{\mX}(E)$ which can be chosen so that $\Omega_\epsilon$ is strictly positive on $\tilde{\mX}$ (because $\mu^*[\Omega]-\epsilon E$ is ample on $\tilde{\mX}$ over $\mathbb{C}$). They then derived the uniform $C^{2}$ estimate on any compact set away from $E$. So the limit $\tilde{\Phi}$ of $\Phi_\epsilon$ is $C^{1,\alpha}$ away from $E$. For details, see \cite{PS09}. %and also Proposition 4.10 in \cite{LiSu}. 
 
Now since $\tilde{\Phi}$ is plurisubharmonic on the (compact and connected) fibers of $\mu: \tilde{\mX}\rightarrow \mX$, it restricts to constant functions on fibers of $\mu$ and hence descends to a bounded $\Omega$-plurisubharmonic function $\Phi$ on $\mX|_{B_1(0)}=\pi^{-1}(B_1(0))$. Moreover, $\Phi$ is $C^{1,\alpha}$ away from $\mX^{sing}=\mu(E)$. 

 \item[{\bf Step 3:}] 
 %This part of arguments needs Berndtsson(-Paun)'s convexity result. 
We choose a Hermitian metric $h_{\Omega}$ on $-K_{\mX/\mathbb{C}}$ such that $-\sddbar\log h_{\Omega}=\Omega$. From the Step 1, $h_{\Omega}^{\otimes m}$ is just a pull back of the Fubini-Study metric on $\mathcal{O}_{\mathbb{P}^N}(1)$. We define
\begin{equation}\label{Dingf(t)}
 f(t)=F_{\Omega|_{\mX_t}}^0(\Phi|_{\mX_t})-\log\left(\int_{\mX_t} e^{-\Phi|_{\mX_t}}dV(h_{\Omega}|_{\mX_t})\right)={\rm I}+{\rm II}.
\end{equation}
Recall that, from Definition \ref{advol}, on each fibre $\mX_t$,
 \[
 dV(h_{\Omega}|_{\mX_t})=(\sqrt{-1})^{n}|v_t^*|_{h_{\Omega}}^{2/m} (v_t\wedge \bar{v}_t)^{1/m},
 \]
where $m$ is a positive integer and $v=(v_t)$ is a generator of $\mathcal{O}(-m K_{\mX/\mathbb{C}})$.  %Again $\mX_0$ being klt is equivalent to saying that  $dV(h_{\Omega}|_{\mX_0})$ is an integrable volume form. 
\begin{remark}
The right hand side of \eqref{Dingf(t)} does not change if replace $\Phi$ by $\Phi+c(w)$ as long as $\Phi+c(w)$ is still an $\Omega$-plurisubharmonic function, where $c(w)$ is any function depending only on the variable $w$ on $B_1(0)$. In particular, we can replace $\Phi$ by $\Phi+c$ for any constant $c$ without changing $f(t)$. Also, by \eqref{transadapt}, we have that when $t=1$,
\begin{eqnarray*}
f(1)&=&F_{\omega_0}^0(\phi)-\log\left(\int_X e^{-\phi} dV(h_{\Omega}|_{\mX_1}) \right)\\
&=&F_{\omega_0}^0(\phi)-\log\left(\frac{1}{V}\int_X e^{-\phi}e^{h_{\omega_0}}\omega_0^n\right)-\log\left(\int_X dV(h_\Omega|_{\mX_1})\right)\\
&=&F_{\omega_0}(\omega_\phi)-\log\left(\int_X dV(h_\Omega|_{\mX_1=X})\right).
\end{eqnarray*}
By scaling $h_{\Omega}$ by a positive constant, we can assume that $\int_X dV(h_\Omega|_{\mX_1})=1$ so that $f(1)=F_{\omega_0}(\omega_\phi)$.
Similarly, we have that
\begin{eqnarray*}
f(0)&=&F_{\Omega|_{\mX_0}}^0(\Phi|_{\mX_0})-\log\left(\int_{\mX_0} e^{-\Phi|_{\mX_0}}dV(h_{\Omega}|_{\mX_0})\right)\\
&=&F^{\mX_0}_{\Omega|_{\mX_0}}((\Omega+\sddbar\Phi)|_{\mX_0})-C(h_\Omega),
\end{eqnarray*}
where we have denoted by $F^{\mX_0}_{\Omega|_{\mX_0}}$ the Ding energy defined on the central fibre satisfying $F^{\mX_0}_{\Omega|_{\mX_0}}(\Omega|_{\mX_0})=0$, and the constant $C(h_\Omega)$ denotes 
\[
C(h_\Omega)=\log\left(\int_{\mX_0}dV(h_{\Omega}|_{\mX_0})\right).
\]
Note that $C(h_\Omega)$ is independent of $\Phi$.
\end{remark}
%where $dV(h_{\Omega}|_{\mX_t})=(\sqrt{-1})^n |\partial_z|_{h_{\Omega}|_{\mX_t}}^2dz\wedge d\bar{z}$ on the smooth fibre. 
%Here, we assume $\mD=\{\mathcal{S}=0\}$ and $|\cdot|$ of $\mathcal{S}$ is the Hermitian metric on $-\lambda K_{\mX/\mathbb{C}}$ with curvature given by $\Omega$.  

Now since $\Phi$ satisfies the homogeneous Monge-Amp\`{e}re equation and by formula \eqref{MAenergy} $F_{\Omega|_{\mX_t}}^0$ is essentially the negative Bott-Chern integral for $c_1(L)^{n+1}$, we get
\[
\sqrt{-1}\partial\bar{\partial}{\rm I}=-\frac{1}{n+1}\int_{\mX_t}(\Omega+\sqrt{-1}\partial\bar{\partial}\Phi)^{n+1}-\Omega^{n+1}=\frac{1}{n+1}\int_{\mX_t}\Omega^{n+1}\ge 0, 
\]
where $\int_{\mX_t}$ denotes the integration along the fibre. The above identities/inequality hold in the sense of pluripotential theory. This can be proved using test functions and approximation argument as in the proof in \cite[Theorem 3.1]{ACKPZ}. Indeed, if $\Phi$ is smooth, we can choose a {\it{negative}} test function $\psi\in \mathcal{E}_0(B_1(0))\cap C(B_1(0))$ with zero boundary values (see \cite{ACKPZ} for the precise definition of
$\mathcal{E}_0(B_1(0))$)  and calculate:
\begin{eqnarray*}
&&\int_{B_1(0)}{\rm I}(\sddbar\psi)=\int_{B_1(0)}\int_{\mX_t}\mathfrak{M}(\sddbar\psi)=\int_{\mX}\mathfrak{M}(\sddbar\psi)\\
&=&\int_{\mX}(\sddbar \mathfrak{M})\psi=-\int_{\mX}\frac{1}{n+1}\sddbar\Phi \sum_{i=0}^n(\Omega+\sddbar\Phi)^{i}\wedge \Omega^{n-i}\cdot \psi\\
&=&-\frac{1}{n+1}\int_{\mX}\psi(
(\Omega+\sddbar\Phi)^{n+1}-\Omega^{n+1})\\
&=&\frac{1}{n+1}\int_{\mX}\psi\Omega^{n+1}=\frac{1}{n+1}\int_{B_1(0)}\psi \int_{\mX_t}\Omega^{n+1}
\le 0.
\end{eqnarray*}
where we have denoted
\[
\mathfrak{M}:=-\frac{1}{n+1}\sum_{i=0}^n \Phi (\Omega+\sddbar\Phi)^{i}\wedge \Omega^{n-i}.
\]
For general $\Phi$ one can use smooth approximations to prove that the above calculation still holds. 

Furthermore, we can verify the continuity of ${\rm I}(t)$ as $t\rightarrow 0$. The sketch of the proof of this fact using the convergence of Monge-Amp\`{e}re measures was given in the second revision of the submitted paper. Recently we noticed that the paper \cite[Proposition 2.19]{SSY} has given a detailed proof along similar line of thoughts. The idea is to decompose the estimate of difference into three parts, which is similar to the estimate \eqref{decompestimate} in the proof of Lemma \ref{IIcont}. In other words, choosing a small
neighborhood $\mathcal{W}(\delta)$ of $\mX^{sing}=\mX_0^{sing}$ (see section 4), we can estimate:
\begin{eqnarray}\label{destimate1}
&&\left|F^0_{\Omega|_{\mX_t}}(\Phi|_{\mX_t})-F^0_{\Omega|_{\mX_0}}(\Phi|_{\mX_0})\right|\\
&\le&\left|\int_{\mX_t\backslash\mathcal{W}(\delta)}\mathfrak{M}|_{\mX_t}-\int_{\mX_0\backslash\mathcal{W}(\delta)}\mathfrak{M}|_{\mX_0}\right|+\left|\int_{\mX_0\cap \mathcal{W}(\delta)}\mathfrak{M}|_{\mX_0}\right|+\left|\int_{\mX_t\cap\mathcal{W}(\delta)} \mathfrak{M}|_{\mX_t}\right|.\nonumber
%&\le& \scriptstyle{\left|\int_{\mX_t\backslash\mathcal{W}(\delta)}e^{-\Phi}dV(h_{\Omega})-\int_{\mX_0\backslash\mathcal{W}(\delta)}e^{-\Phi}dV(h_{\Omega})\right|+e^{\|\Phi\|_{L^{\infty}}}\left(\int_{\mX_t\cap\mathcal{W}(\delta)}dV(h_{\Omega})+\int_{\mX_0\cap \mathcal{W}(\delta)}dV(h_{\Omega})\right)}.
\end{eqnarray}
For the second term on the right-hand-side of \eqref{destimate1}, we can estimate:
\begin{eqnarray*}
\left|\int_{\mX_0\cap \mathcal{W}(\delta)}\frak{M}|_{\mX_0}\right|\le\frac{\|\Phi\|_{L^{\infty}}}{n+1}\sum_{i=0}^n\int_{\mX_0\cap\mathcal{W}(\delta)}(\Omega+\sddbar\Phi)^i\wedge\Omega^{n-i}.\end{eqnarray*}
By choosing $\delta$ sufficiently small, each term in the above summation can be made arbitrarily small:
\begin{equation}\label{2small}
\int_{\mX_0\cap\mathcal{W}(\delta)}(\Omega+\sddbar\Phi)^i\wedge\Omega^{n-i}\le \epsilon \ll 1.
\end{equation}
This is because that $\mX_0^{sing}$ is a pluripolar set and carries no Monge-Amp\`{e}re mass with bounded potential. To estimate the third term on the right-hand-side of \eqref{destimate1}, we first notice that, because $\Phi$ is continuous on $\mX\backslash \mathcal{W}(\delta)$, essentially by the convergence of Monge-Amp\`{e}re measures 
(see e.g. \cite[Corollary 3.6, Chapter 3]{Dem} and \cite{Xing}) we have:
\begin{eqnarray}
\hskip -3mm\lim_{t\rightarrow 0}\int_{\mX_t\backslash\mathcal{W}(\delta)}(\Omega+\sddbar\Phi)^i\wedge\Omega^{n-i}&=&\int_{\mX_0\backslash\mathcal{W}(\delta)}(\Omega+\sddbar\Phi)^i\wedge\Omega^{n-i}\label{outest1}\\
\!\!&=&\!\!{\rm Vol}(\mX_0)-\int_{\mX_0\cap\mathcal{W}(\delta)}(\Omega+\sddbar\Phi)^i\wedge\Omega^{n-i}\label{outest2}\\
&\ge&{\rm Vol}(\mX_0)- \epsilon,\label{outest3} 
\end{eqnarray}
where ${\rm Vol}(\mX_0)$ is the same as $V$ in Definition \ref{volume}. The last inequality in \eqref{outest3} is because of \eqref{2small}.  As suggested by the referee, the equality in \eqref{outest1} can be argued as follows. Choose a partition of unity
$\{\rho_\alpha\}$ subordinate to a covering $\{\mathcal{U}_\alpha\}$ of $\mX\backslash \mathcal{W}(\delta)$. Note that the earlier exact choice of $\mathcal{W}(\delta)$ is not essential as long as $\mathcal{W}(\delta)$ is sufficiently small. So now by appropriately modifying $\mathcal{W}(\delta)$, each $\mathcal{U}_\alpha$ can be chosen using local holomorphic coordinates so that $\mathcal{U}_\alpha$ is biholomorphic to a polydisc $B_r(0)\times \mathbb{D}^n$ in $\mathbb{C}^{n+1}$ and the projection $\mathcal{U}_\alpha\rightarrow B_r(0)$ is a local product fibration over $B_r(0)$ for $r\ll 1$. 
Now we claim that 
%So we can use the convergence of Monge-Amp\`{e}re measures mentioned above and Fatou's Lemma to get (noting that $\rho_\alpha\in C^{\infty}_0(\mathcal{U}_\alpha)$)
\begin{equation}\label{uconv}
\lim_{t\rightarrow 0}\int_{\mathcal{U}_\alpha\cap \mX_t}\rho_\alpha (\Omega+\sddbar\Phi)^i\wedge \Omega^{n-i}=\int_{\mathcal{U}_\alpha\cap \mX_0}\rho_\alpha(\Omega+\sddbar\Phi)^i\wedge
\Omega^{n-i}.
\end{equation}
Then by patching together these local convergences using the property of the partition of unity, we get the convergence stated in \eqref{outest1}.

To see why \eqref{uconv} holds, let's first call $\rho_\alpha^{t}:=\rho_\alpha|_{\mathcal{U}_\alpha\cap \mX_t}$ and $\mathfrak{M}_i^{t}:=(\Omega+\sddbar\Phi)^i\wedge \Omega^{n-i}|_{\mX_t}$. Because we have chosen $\mathcal{U}_\alpha \cong B_r(0)\times \mathbb{D}^n$ where $\mathbb{D}^n\subset \mathbb{C}^n$ is a fixed polydisc, we can estimate the difference of the left-hand-side and right-hand-side of \eqref{uconv} as:
\begin{equation}\label{udiff}
\left|{\rm left}-{\rm right}\right|\le \left|\int_{\mathbb{D}^n}(\rho^{t}_\alpha-\rho^{0}_\alpha)\mathfrak{M}^t_i\right|+\left|\int_{\mathbb{D}^n}\rho_\alpha^0(\mathfrak{M}^t_i-\mathfrak{M}^0_i)\right|.
\end{equation}
The first term on the right-hand-side of \eqref{udiff} is estimated from above by:
\begin{equation}\label{udiff1}
\left\|\rho_\alpha^t-\rho_\alpha^0\right\|_{L^\infty}\int_{\mathbb{D}^n}\mathfrak{M}^t_i\le \|\rho^t_\alpha-\rho^0_\alpha\|_{L^\infty}{\rm Vol}(\mX_t)=\|\rho^t_\alpha-\rho_\alpha^0\|_{L^\infty}V,
\end{equation}
where $V=(2\pi)^n (-K_{\mX_t})^n$ (see Definition \ref{volume}) is independent of $t$. Noting that $\rho_\alpha\in C^{\infty}_0(\mathcal{U}_\alpha)$, when $t$ is sufficiently close to $0$, the last term in \eqref{udiff1} can indeed be made arbitrarily small. The second term on the right-hand-side of \eqref{udiff} converges to zero as $t\rightarrow 0$ because of the (weak) convergence of Monge-Amp\`{e}re measures mentioned above.

So by \eqref{outest1}-\eqref{outest3}, when $t$ is sufficiently close to $0$, we have:
\[
\int_{\mX_t\backslash \mathcal{W}(\delta)}(\Omega+\sddbar\Phi)^i\wedge\Omega^{n-i}\ge {\rm Vol}(\mX_0)-2\epsilon.
\]
For such $t$, we can thus estimate the third term in \eqref{destimate1} (cf. \cite[Proof of Proposition 2.19]{SSY}):
\begin{eqnarray*}
\left|\int_{\mX_t\cap \mathcal{W}(\delta)}\mathfrak{M}|_{\mX_t}\right|&\le& \frac{\|\Phi\|_{L^{\infty}}}{n+1}\sum_{i=0}^n\int_{\mX_t\cap \mathcal{W}(\delta)}(\Omega+\sddbar\Phi)^i\wedge \Omega^{n-i}\\
&=&\frac{\|\Phi\|_{L^{\infty}}}{n+1}\sum_{i=0}^n \left({\rm Vol}(\mX_t)-\int_{\mX_t\backslash\mathcal{W}(\delta)}(\Omega+\sddbar\Phi)^i\wedge\Omega^{n-i}\right)\\
&\le &\frac{\|\Phi\|_{L^\infty}}{n+1}\sum_{n=0}^n\left({\rm Vol}(\mX_t)-({\rm Vol}(\mX_0)-2\epsilon)\right)\\
&=&2\epsilon \|\Phi\|_{L^\infty}.
\end{eqnarray*}
Here we have used that fact that ${\rm Vol}(\mX_t)=(2\pi)^n(-K_{\mX_t})^n=(2\pi)^n(-K_{\mX_0})^n={\rm Vol}(\mX_0)$.
Once we have chosen the $\delta>0$ and $|t|$ sufficiently small such that the estimates for the second term and the third term in the right-hand-side of \eqref{destimate1} are sufficiently small as above, the first term in \eqref{destimate1} can also be made arbitrarily small as long as $t$ is sufficiently close to $0$, again by the convergence of Monge-Amp\`{e}re measures away from $\mathcal{W}(\delta)$ and the argument involving the partition of unity as above.
%because $\Phi$ is uniformly continuous away from $\mathcal{W}(\delta)$,   by the result on . 
%Choosing a small neighborhood $\mathcal{W}(\delta)$ of $\mX^{sing}$ (see section \ref{sectionlemma}), we can estimate:
%Then on the one hand, by choosing $\delta$ sufficiently small, one can make the second 
%and third terms arbitrarily small. 

\begin{remark}
One referee pointed out that %the following results. When $\Phi$ is smooth, the continuity of ${\rm I}(t)$ as $t\rightarrow 0$ was proved in \cite{Mor99}. For the more general $\Phi$ considered here, 
Berman \cite{Berm2} has already proved that ${\rm I}(t)$ is continuous on $B_1(0)^*$ and is bounded on $B_1(0)$, and that his proof can also be modified to show that ${\rm I}(t)$ is continuous at $t=0$. Berman's proof used the fact that the functional $F^0_{\omega}(\phi)$ can be written as a difference of
metrics on the corresponding Deligne pairing which is continuous by a result of Moriwaki. The use of Moriwaki's result in this context was first observed in \cite{PRS}.
\end{remark}
 %whose proof can also be modified to show ${\rm I}(t)$ is continuous.

Next we want to prove that ${\rm II}$ is also subharmonic and continuous at $t=0$. This part is more difficult and we need Berndtsson's subharmonicity result that we now recall. %Berndtsson-Paun's result in \cite[Theorem 0.1]{BePa}.
Assume that $\mathcal{H}$ is any possibly singular Hermitian metric on a relative ample line bundle $\mathcal{L}\rightarrow \mX$. The relative Bergman kernel metric on $\mathcal{O}_{\mX}(K_{\mX/\mathbb{C}}+\mL)$ is defined by:
\begin{equation}\label{BKmetric}
|s|_{BK}^2=\frac{|s|^2}{\sum_{i}|s_i|^2}
\end{equation}
where $\{s_i\}$ is an $L^2$-orthonormal basis of $H^0(\mX_t, K_{\mX_t}+\mL|_{\mX_t})$ under the $L^2$-inner product induced by $\mathcal{H}$. Berndtsson's fundamental result in \cite{Bern1} (see also \cite[Theorem 0.1]{BePa}) says that, if the metric $\mathcal{H}$ on $\mL\rightarrow \mX$ has positive curvature current: $-\sddbar\log \mathcal{H}\ge 0$, then $|s|^2_{BK}$ has a positive curvature on $\mX^{*}$ where the projection $\pi: \mX^{*}\rightarrow \mathbb{C}^*$ is a smooth fibration. In other words, under the above assumption, we have:
\[
-\sddbar\log|\cdot|_{BK}^2\ge 0 \mbox{ on } \mX^*.
\] 
Note that a priorly we don't know what happens on the whole $\mX$.
% the best we can get by \cite{BePa} is that ${\rm II}$ is bounded from above, and so it extends to a plurisubharmonic function on $\mX$. We will show in Lemma \ref{IIcont} that ${\rm II}$ actually subharmonic on t

Berndtsson's result was applied in the current set-up in the work of \cite{Bern} and \cite[Lemma 6.5]{BBGZ}. To do this, we write $\mathcal{O}_{\mX}$ as $K_{\mX}+(-K_{\mX})$ and take the Hermtian metric to be $\mathcal{H}=h_{\Omega} e^{-\Phi}$. Note that $H^0(\mX_t, \mathcal{O}_{\mX_t})\cong \mathbb{C}$ is a 1-dimensional vector space spanned by the constant section ${\bf 1}$. The $L^2$-norm of the constant section ${\bf 1}$ is equal to:
\begin{equation}\label{1L2norm}
\|{\bf 1}\|_{L^2}^2=(\sqrt{-1})^n\int_{\mX_t} |v_t^*|_{h_\Omega}^{2/m} e^{-\Phi} (v_t\wedge \bar{v}_t)^{1/m}=\int_{\mX_t} e^{-\Phi|_{\mX_t}}dV(h_{\Omega|_{\mX_t}}).
\end{equation}
So by \eqref{BKmetric} the relative Bergman metric on $\mathcal{O}_{\mX}$ is defined for any holomorphic function $h\in\mathcal{O}_{\mX}$ as:
\[
|h|_{BK}^2=\frac{|h|^2}{|{\bf 1}|^2/\|{\bf 1}\|_{L^2}^2}=|h|^2 \|{\bf1}\|_{L^2}^2.
\]
So by \eqref{1L2norm} we see that:
\[
{\rm II}(t)=-\log\left(\int_{\mX_t}e^{-\Phi|_{\mX_t}} dV(h_{\Omega|_{\mX_t}})\right)=-\log |1|_{BK}^2.
\]
In particular, the right-hand side is a pull-back function from the base.
%it's easy to see that ${\rm II}(t)=-\log \|1\|_{BK}^2$ where $\|1\|_{BK}^2$ is the relative Bergman metric on $\pi_*(\mathcal{O}_{\mX})=\mathcal{O}_{\mathbb{C}}$.  
Now the Hermitian metric $\mathcal{H}=h_{\Omega} e^{-\Phi}$ on $-K_{\mX}$ satisfies
\begin{equation}\label{positivecurvature}
-\sqrt{-1}\partial\bar{\partial}\log\left(e^{-\Phi}dV(h_{\Omega})\right)=\Omega+\sqrt{-1}\partial\bar{\partial}\Phi\ge 0.
\end{equation}
So Berndtsson's result (\cite[Lemma 6.5]{BBGZ}) implies that ${\rm II}(t)$ is subharmonic on
$B_1(0)^*:=B_1(0)\backslash \{0\}$. %Moreover, by \cite{BePa} again, ${\rm II}$ is bounded from above, and so by standard potential theory, it extends to a subharmonic function on $B_1(0)$.  %$\sqrt{-1}B_1(0) {\rm II} dz\wedge d\bar{z}$ extends to a positive current on $B_1(0)$. 
Now we need the following Lemma which is equivalent to Lemma \ref{contlem} and will be proved in the next section:
\begin{lem}[Lemma \ref{contlem}]\label{IIcont}
The part ${\rm II}(t)$ is continuous as a function of $t$. In other words, we have the convergence:
\begin{equation}\label{globalconvergence1}
\lim_{t\rightarrow 0}\int_{\mX_t}e^{-\Phi|_{\mX_t}} dV(h_{\Omega}|_{\mX_t})=\int_{\mX_0} e^{-\Phi} dV(h_{\Omega}).
\end{equation}
\end{lem}
Assuming Lemma \ref{IIcont}, by standard potential theory, ${\rm II}(t)$ is subharmonic on the whole $B_1(0)$: $\Delta {\rm II}\ge 0$, because it coincides with its subharmonic extension. So by the maximal principle and $S^1$-symmetry, we have
\[
\max_{t\in \partial B_1(0)}{\rm II}(t)={\rm II}(1)\ge {\rm II}(0).
\]
Combining the above discussion, we see that $f(t)$ is subharmonic on $B_1(0)$ and so by the maximal principle, 
\begin{equation}\label{lowerineq1}
f(1)=F_{\omega_0}(\omega_\phi)\ge f(0)=F^{\mX_0}_{\Omega|_{\mX_0}}(\Omega_{\Phi}|_{\mX_0})-C(h_\Omega).
\end{equation}

\item[{\bf Step 4:}]
 Because we assume that $\mX_0$ has a (weak) K\"{a}hler-Einstein $\omega_{KE}^{\mX_0}$, $\omega_{KE}^{\mX_0}$ obtains the minimum of Ding-energy functional: 
  \begin{equation}\label{lowerineq2}
 F_{\Omega|_{\mX_0}}^{\mX_0}(\Omega_\Phi|_{\mX_0})\ge F^{\mX_0}_{\Omega|_{\mX_0}}(\omega_{KE}^{\mX_0}).
 \end{equation}
This was proved in \cite{DT93} (see also \cite{Ti99}) in the smooth case, and was generalized to the $\mathbb{Q}$-Fano case in \cite{BBEGZ}. Also see discussions in \cite{LiSu}. 
 
\end{enumerate}
Now Theorem \ref{bdDing} follows by combining inequalities \eqref{lowerineq1} and \eqref{lowerineq2}.
 
\end{proof}
\begin{remark}\label{logmod}
Here we point out the necessary modifications of the above arguments for the logarithmic case, i.e. for the pair $(\mX,\alpha\mD)$. Denote $\beta=1-\alpha$. Step 1 and Step 2 stay the same. In Step 3, we need to consider the (relative) log volume form:
\begin{equation}\label{logvolumeform}
%\scriptstyle
dV\!\!\left((\mX_t,(1-\beta) \mD_t); h_{\Omega}e^{-\Phi}|_{\mX_t}\right)\!=\!\left.\frac{e^{-\Phi}}{(|\mathcal{S}|_{h_\Omega}^2e^{-\lambda\Phi})^{1-\beta}}\right|_{\mX_t}\!\!dV(h_{\Omega}|_{\mX_t})\!=\!\left.\frac{e^{-r(\beta)\Phi}}{|\mathcal{S}|^{2(1-\beta)}_{h_\Omega}}\right|_{\mX_t}\!\!dV(h_\Omega|_{\mX_t}).
\end{equation}
where $r(\beta)=1-(1-\beta)\lambda$, $\mathcal{S}$ is the holomorphic defining section of $\mD\subset |-\lambda K_{\mX/\mathbb{C}}|$ and $|\cdot |_{h_\Omega}^2$ is the naturally induced Hermitian metric on $-\lambda K_{\mX/\mathbb{C}}$ by the hermitian metric 
$h_{\Omega}$ from the beginning of Step 3. Then \eqref{positivecurvature} becomes
\begin{eqnarray*}
-\sqrt{-1}\partial\bar{\partial}\log\left(e^{-r(\beta)\Phi}\frac{dV(h_{\Omega})}{|\mathcal{S}|_{h_\Omega}^{2(1-\beta)}}\right)&=&{\scriptstyle r(\beta)\sddbar\Phi-\sddbar\log dV(h_{\Omega})+(1-\beta)\sddbar\log |\mathcal{S}|_{h_\Omega}^2}\\
&=&\Omega+r(\beta)\sqrt{-1}\partial\bar{\partial}\Phi+(1-\beta)(-\lambda \Omega+\{\mathcal{S}=0\})\\
&=&(1-(1-\beta)\lambda)(\Omega+\sddbar\Phi)+(1-\beta)\{\mathcal{S}=0\}.
\end{eqnarray*}
\end{remark}
Note that we can assume $r(\beta)\ge 0$. Otherwise, as has been pointed out in the proof of Proposition \ref{PDEcont}, the $\alpha$-invariant or log-$\alpha$-invariant (\cite{Berm}, \cite{LiSu}) is sufficient to prove the lower boundedness (even the properness) of (log-)Mabuchi energy and hence the
lower boundedness of (log-)Ding-energy. See \cite{BBEGZ} and \cite{LiSu}.  When $r(\beta)\ge 0$, the right-hand-side is a positive current and we can again use Berndtsson's subharmoncity result recalled earlier plus the Remark \ref{logmodcont}. Finally, the result in Step 4 in the log setting has been proved in \cite{BBEGZ}.

\section{Proof of Lemma \ref{IIcont}}\label{sectionlemma}

%\begin{proof}[Proof of Lemma \ref{IIcont}]
%Note that Lemma \ref{contlem} is equivalent to Lemma \ref{IIcont}.
First note that by the expression of ${\rm II}$ in \eqref{Dingf(t)}, ${\rm II}(t)$ is continuous on $B_1(0)^*$ since $\Phi$ is continuous on $\mX^*:=\mX\backslash \mX_0$ and $dV(h)$ is a smooth volume form on $\mX^*$. Recall that we want to prove the convergence:
\begin{equation}\label{globalconvergence}
\lim_{t\rightarrow 0}\int_{\mX_t}e^{-\Phi|_{\mX_t}} dV(h_{\Omega}|_{\mX_t})=\int_{\mX_0} e^{-\Phi|_{\mX_0}} dV(h_{\Omega}|_{\mX_0}).
\end{equation}
Again the finiteness on the right hand side follows from $\mX_0$ being klt. %For convenience, we assume $\mX$ has been embedded into $\mathbb{P}^N\times\mathbb{C}$, and for any $y\in\mathbb{P}^N\times \mathbb{C}$ we denote $\mathcal{B}(y,\delta)$ to be the ball of radius $\delta$ around $y$ under the product Riemannian metric on $\mathbb{P}^N\times\mathbb{C}$. We want to reduce proving global convergence \eqref{globalconvergence} to proving local volume convergence. More precisely, for any point $x_0\in \mX_0$ and any $0<\delta\ll 1$ small, we will construct a $\mathcal{U}(x_0,\delta)$ which is a neighborhood of $x_0$ in $\mX$ such that
For convenience, we assume $\mX$ has been embedded into $\mathbb{P}^N\times\mathbb{C}$, and we denote by $\mathcal{B}(\mX^{sing},\delta)$ the tube of radius $\delta$ around the closed set $\mX^{sing}$ in the metric product $\mathbb{P}^N\times\mathbb{C}$. We want to reduce proving global convergence \eqref{globalconvergence} to proving local volume convergence. More precisely, we will construct a family of neighborhoods $\{\mathcal{W}(\delta)\}_{\delta\in (0,1)}$ of $\mX^{sing}$ in $\mX$ in analytic topology, such that the following conditions are satisfied:
\begin{enumerate}
\item
If $\delta_1<\delta_2$, then $\overline{\mathcal{W}({\delta_1})}$ is compact in $\mathcal{W}({\delta_2})$.
\item There exists a fixed $\Lambda>1$ such that $\scriptstyle{\mathcal{W}({\delta})\subset \mathcal{B}(\mX^{sing},\Lambda\delta)}$. As a consequence, $\scriptstyle{\bigcap_{\delta>0}\mathcal{W}(\delta)=\mX^{sing}}$.
\item Local volume convergence holds for any $\mW(\delta)$:
\begin{equation}\label{locvolconv}
\lim_{t\rightarrow 0}\int_{\mX_t\cap\mW(\delta)} dV(h_{\Omega}|_{\mX_t})=\int_{\mX_0\cap\mW(\delta)}dV(h_{\Omega}|_{\mX_0}).
\end{equation}
\end{enumerate}
%In (c) and (d) above, for simplicity of notation, we have denoted:
%\[
%\mathcal{W}_0(\delta)=\mX_0\cap\mW(\delta), \quad \mW_t(\delta)=\mX_t\cap\mW(\delta).
%\]
Assuming that we have constructed such a family of $\mW(\delta)$, we claim that we can prove \eqref{globalconvergence}. To see this, we estimate the difference of integrals in \eqref{globalconvergence}%we split the integral in part $\rm II$ as:
%\begin{eqnarray*}
%\int_{\mX_t} e^{-\Phi} dV(h_{\Omega})=\int_{\mX_t\backslash\mathcal{W}_\delta}e^{-\Phi}dV(h_{\Omega})+\int_{\mX_t\cap\mathcal{W}_\delta}e^{-\Phi}dV(h_{\Omega}).
%\end{eqnarray*}
\begin{eqnarray}\label{decompestimate}
&&\scriptstyle{\left|\int_{\mX_t}e^{-\Phi} dV(h_{\Omega})-\int_{\mX_0}e^{-\Phi}dV(h_{\Omega})\right|}\\
&\le& \scriptstyle{\left|\int_{\mX_t\backslash\mathcal{W}(\delta)}e^{-\Phi}dV(h_{\Omega})-\int_{\mX_0\backslash\mathcal{W}(\delta)}e^{-\Phi}dV(h_{\Omega})\right|+\left|\int_{\mX_t\cap\mathcal{W}(\delta)} e^{-\Phi}dV(h_{\Omega})-\int_{\mX_0\cap \mathcal{W}(\delta)}e^{-\Phi}dV(h_{\Omega})\right|}\nonumber\\
&\le& \scriptstyle{\left|\int_{\mX_t\backslash\mathcal{W}(\delta)}e^{-\Phi}dV(h_{\Omega})-\int_{\mX_0\backslash\mathcal{W}(\delta)}e^{-\Phi}dV(h_{\Omega})\right|+e^{\|\Phi\|_{L^{\infty}}}\left(\int_{\mX_t\cap\mathcal{W}(\delta)}dV(h_{\Omega})+\int_{\mX_0\cap \mathcal{W}(\delta)}dV(h_{\Omega})\right)}\nonumber.
\end{eqnarray}
For fixed small $\delta$,  the first term is small when $t$ is sufficiently small because $\Phi$ is uniformly continuous on $\mX|_{B_1(0)}\backslash\mathcal{W}_{\delta}$ by the $C^{1,\alpha}$-regularity result of Phong-Sturm \cite{PS09} recalled in Step 2 of Section \ref{sectionthm} and $dV(h_{\Omega})$ is a smooth volume form on $\mX\backslash\mathcal{W}_{\delta}$. For the terms in the bracket, first note that $\|\Phi\|_{L^{\infty}}$ is finite since $\Phi$ is uniformly bounded on $\mX|_{B_1(0)}$ by the discussion above in Step 2. Secondly, the volume of $\mW(\delta)\cap\mX_0$ is negligible when $\delta$ is small because $dV(h_{\Omega})$ is an $L^p$-volume form on $\mX_0$ for some $p>1$ by the klt property of $\mX_0$. In other words, 
\begin{equation}\label{negligible}
\lim_{\delta\rightarrow 0}\int_{\mX_0\cap\mW(\delta)}dV(h_{\Omega}|_{\mX_0})=\int_{\mX^{sing}_0} dV(h_\Omega|_{\mX_0})=0.
\end{equation}
So the convergence identities \eqref{locvolconv} and \eqref{negligible} imply that the two volume integrals in the last bracket can be arbitrarily small when $t$ and $\delta$ are sufficiently small. So the claim follows.

%Now we prove the (local) volume convergence. 

%For any $x\in\mX_0$, choose a small neighborhood $\mathcal{U}$ of $x$ in $\mX$. Choose a local generator $v=(v_t)$ of $\mathcal{O}(m K_{\mX/\mathbb{C}})(\mathcal{U})$. Since $h_{\Omega}$ is the pull back of Fubini-Study metric which is smooth, the density function $|v_t^*|^2_{h_{\Omega}|_{\mX_t}}$ uniformly converges to $|v_0^*|^2_{h_{\Omega}|_{\mX_0}}$, we only need to prove the convergence:
%\begin{equation}\label{locvolconv}
%\lim_{t\rightarrow 0}\int_{\mU_t}(v\wedge\bar{v})^{1/m}=\int_{\mU_0}(v_0\wedge\bar{v}_0)^{1/m}, 
%\end{equation}
%where $\mU_t=\mU\cap\mX_t$. We will prove this by calculating the pull-back of the integrals to a log resolution. Note first that in the special degeneration $\mX$, $\mX_0=\{t=0\}$ is a reduced fibre. 
%This should follow directly from the construction of $\mX$ in Theorem \ref{specialdeg}. Or we can use the fact that the generalized Futaki invariant of $\mX$ is zero. If $\mX_0$ was not reduced, we could do the base change and get a special degeneration with negative Futaki invariant, which is in contradiction with $X$ being K-semistable. For details, see \cite[Claim 1]{LX11}. 
Now we construct $\mW(\delta)$ which is a neighborhood of $\mX^{sing}$. %The idea to construct it is to first construct a neighborhood of exceptional divisors of a log resolution. 
%\[
%\mW(\delta)=\bigcup_{\tilde{x}\in \cup_{i=1}^{K}E_i}\mu(\tilde{\mU}(\tilde{x},\delta)).
%\]
We choose a log resolution of the pair $(\mX, \mX_0)$: $\mu: \tilde{\mX}\rightarrow \mX$ and denote $\tilde{\pi}=\pi\circ\mu$. So we have the commutative diagram:
\begin{equation}\label{commut}%\centerline{
\xymatrix{
  \mX_0' \ar[dr]_{\mu|_{\mX'_0}} \ar[r]^{\subset} &\tilde{\mX}_0 \ar[d]^{} \ar@{^{(}->}[r]_{} &
                    \tilde{\mX} \ar[d]_{\mu} \ar@/^1.5pc/[dd]^{\tilde{\pi}} \\
  & \mX_0 \ar[d]^{} \ar@{^{(}->}[r]_{} & \mX  \ar[d]_{\pi}          \\
  & \{0\}\ar@{^{(}->}[r]^{}     & \mathbb{C}        }
%}
\end{equation}
Note that in the special degeneration $\mX$, $\mX_0=\{t=0\}$ is a reduced fibre. Then we have
\begin{equation}\label{tilde0}
\tilde{\mX}_0=\mu^*\mX_0=\mu^*\pi^*(\{t=0\})=\tilde{\pi}^*(\{t=0\})=\mX_0'+\sum_{i=1}^{K}a_i E_i, \mbox{ with } \mathbb{Z}\ni a_i>0,
\end{equation}
where $\mX_0'$ is the strict transform of $\mX_0$ under $\mu^{-1}$. $E_i$'s are exceptional divisors. The divisors $\mX_0'$ and $E_i$'s have simple normal crossings. For any $\tilde{x}\in \tilde{\mX}_0=\mX_0'\bigcup \cup_{i=1}^{K} E_i$, we will construct a neighborhood $\tilde{\mU}(\tilde{x},\delta)$ in $\tilde{\mX}$ using local normal crossing coordinates. The key property satisfied by $\tilde{\mU}(\tilde{x},\delta)$ is what we will call the {\em local volume convergence} property, which is the following equality corresponding to \eqref{locvolconv}:
\begin{equation}\label{localconvprop}
\lim_{t\rightarrow 0}\int_{\tilde{\mX}_t\cap \tilde{\mU}(\tilde{x},\delta)}\mu^*(v\wedge \bar{v})=\int_{\mX_0'\cap\tilde{\mU}(\tilde{x},\delta)}\mu^*(v\wedge\bar{v}).
\end{equation}
Note that if $\mX_0'\cap\tilde{\mU}(\tilde{x},\delta)$ is an empty set, then the right hand side is equal to zero. So \eqref{localconvprop} essentially says that the limit of the pull-back of volume integrals is concentrated on $\mX'_0$.

Assuming we have achieved this, we just define our desired $\mW(\delta)$ on $\mX$ to be
\begin{equation}\label{Xsingnbhd}
\mW(\delta)=\bigcup_{\tilde{x}\in \cup_{i=1}^{K}E_i}\mu(\tilde{\mU}(\tilde{x},\delta)).
\end{equation}
Note that on the right hand side, we can choose a finite sub covering by the compactness of $\tilde{\mX}_0$. By the construction of $\tilde{\mU}(\tilde{x},\delta)$, it is then easy to verify that $\mW(\delta)$ satisfies our requirements.
%To simplify the notation, we will write:
%\[
%\tilde{\mU}_t(\tilde{x},\delta)=\tilde{\mU}(\tilde{x},\delta)\cap \tilde{\mX}_t.
%\]

Before we construct the neighborhoods and calculate, we need to recall an important result from complex algebraic geometry called {\it inversion of adjunction}. For this, we consider the usual formula defining the discrepancy $b_i=a(E_i, \mX, \mX_0)$ of $E_i$ with respect $(\mX,\mX_0)$ (\cite[Definition 2.25]{KM98}):
\begin{equation}\label{logres}
K_{\tilde{\mX}/\mathbb{C}}+\mX_0'=\mu^*(K_{\mX/\mathbb{C}}+\mX_0)-\sum_{i=1}^K b_i E_i.
\end{equation}
Because $\mX_0$ is klt, the inversion of adjunction \cite[Theorem 5.50]{KM98} says that $(\mX, \mX_0)$ is plt, which implies in particular $b_i<1$. 
This will be the key fact for us to estimate the integrals. See the discussion in Remark \ref{remnbhd}.
Combining equation \eqref{tilde0} and \eqref{logres}, we get
\begin{equation}\label{pullvol}
\mu^*K_{\mX/\mathbb{C}}=K_{\tilde{\mX}/\mathbb{C}}-\sum_{i=1}^K(a_i-b_i)E_i.
\end{equation}
Also note that by adjunction formula, from \eqref{logres} we have
\begin{equation}\label{adjunction}
K_{\mX_0'}=(\mu|_{\mX_0'})^*K_{\mX_0}-\sum_{i=1}^{K}b_i E_i|_{\mX_0'}.
\end{equation}
with $b_i<1$, which is saying exactly that $\mX_0$ is klt since $\mu|_{\mX'_0}: \mX'_0\rightarrow \mX_0$ is a resolution of singularities.

\begin{remark}\label{remnbhd}
The following calculations are separated into two cases depending on whether $\tilde{x}$ is contained in $\mX_0'$ (the strict transform of $\mX_0$ under resolution) or not. In both cases, there are essentially two steps. For case 1 when $\tilde{x}\in\mX_0'$, first we reduce the local volume integral on $\mU_{t}(\tilde{x},\delta)$ to the integral on the image of the projection of $\mU_{t}(\tilde{x},\delta)$ to $\mX_0'$. Secondly, as $t\rightarrow 0$, we show that these integrals converge to the volume integral on $\mX'_0\cap \mU(\tilde{x},\delta)$ because the domains of integrals converge and the positive integrands also converge under domination. For case 2 when $\tilde{x}\not\in\mX_0'$, we first project the local volume integral to an integral on some appropriately chosen exceptional divisor. Secondly, as $t\rightarrow 0$, we estimate the volume integrals to show that they actually converge to zero. In the estimates in both case, we use essentially the property that $(\mX,\mX_0)$ has plt singularities, i.e. discrepancies of the exceptional divisors over $\mX$ are bigger than $-1$. This important property is called the inversion of adjunction from birational algebraic geometry. 

By this argument we see that the choice of $\tilde{\mU}(\tilde{x},\delta)$ is very flexible. Actually any small simply-connected neighborhood $\tilde{\mU}(\tilde{x},\delta)$ of $\tilde{x}$ satisfies the local convergence property in \eqref{localconvprop}. For example, we may well choose the balls:
\[
\tilde{\mU}(\tilde{x}, \delta)=\left\{\sum_{i=0}^{n}|w_i|^2\le \delta^2\right\},
\]
under local coordinates adapted to the simple normal crossing singularities.
However for the simplicity of calculations, we will choose a small polydisk around $\tilde{x}$ in \eqref{boxnbhd}. %The same remark applies to the following case of calculations.
The above ideas of the calculation are illustrated by the figure 1. 
%\vskip -2mm
%\begin{figure}[h]
%  \begin{center}
%  \subfigure[$\tilde{\mX}$]{\label{ontmX}\includegraphics[height=5cm]{Slide2.jpg}}
%    \subfigure[$\mX$]{\label{onmX}\includegraphics[height=5cm]{Slide3.jpg}}  
%  \end{center}
%  \caption{Calculate via resolution}
%  \label{resolution}
%\end{figure}
\begin{figure}[h]\label{presolution}
\includegraphics[height=5cm]{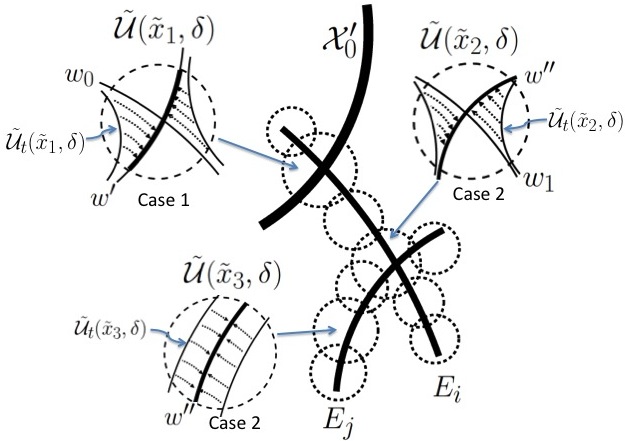}
%\caption{$\tilde{\mX}$}
\hskip -1mm
\includegraphics[height=4.5cm]{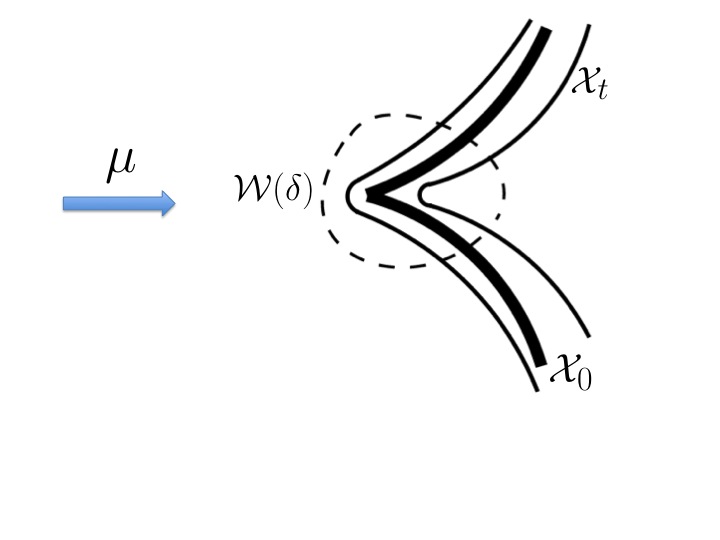}
\label{resolution}
\caption{$\mu: \tilde{\mX}\longrightarrow\mX$}
\end{figure}
\end{remark}
%\begin{claim}
%To prove the convergence \eqref{locvolconv}, we just need to prove that for any point $\tilde{x}\in \mu^{-1}(\mU)\subset \tilde{\mX}_0=\mu^{-1}(\mX_0)=\tilde{\pi}^{-1}(\{0\})$, there exists a neighborhood $\tilde{U}(\tilde{x})$ such that
%\end{claim}
%Choose a small neighborhood $\tilde{\mU}$ of $\tilde{x}$ in $\tilde{\mX}$ and write $\tilde{\mU}_t=\tilde{\mU}\cap\tilde{\mX}_t$. Note that when $t\neq 0$, $\tilde{\mX}_t=\tilde{\pi}^{-1}(\{t\})=\pi^{-1}(\{t\})=\mX_t$ because $\mu$ is isomorphism over $\mX^*$. We estimate the pull-back integral on $\tilde{\mU}_t$.
Now we can start to calculate the volume integrals on small neighborhoods of any $\tilde{x}\in \mX'_0\bigcup \cup_{i=1}^{K}E_i$. The following arguments of calculations are basically detailed explanations of figure 1.
\begin{enumerate}
\item[Case 1:] $\tilde{x}\in \mX_0'$. Without loss of generality, we can assume
\[
\tilde{x}\in \mX_0'\bigcap \cap_{i=1}^{N_{\tilde{x}}}E_i.
\]
Choose a coordinate chart $w=\{w_0,w_1,\dots,w_n\}=\{w_0, w'\}$ such that  $w_i(\tilde{x})=0$ for $i=0,1,\dots, n$, and locally $\mX_0'=\{w_0=0\}$ and $E_j=\{w_j=0\}$ for $1\le j\le n$.  By \eqref{tilde0} we can assume that the map $\tilde{\pi}$ (see the diagram \eqref{commut}) is given by 
\begin{equation}\label{monoproj}
t=w_0\prod_{i=1}^{N_{\tilde{x}}} w_i^{a_i}.
\end{equation}
with $a_i\ge 1$ for $i=1,\dots, N_{\tilde{x}}$. These can be done because $\mX'_0\bigcup \cup_{i=1}^{N_{\tilde{x}}}E_i$'s has simple normal crossings. Note that, in particular, $w'=\{w_1,\dots, w_n\}$ is a local coordinate system on $\mX'_0$ near $\tilde{x}$. We now consider the region:
\begin{equation}\label{boxnbhd}
\tilde{\mU}(\tilde{x},\delta)=\{|w_j|\le \delta, j=0,\dots, n\}.
\end{equation}
Note that, by the previous remark \eqref{remnbhd}, this specific choice is to make calculations simple. Any comparable choice of $\tilde{\mU}(\tilde{x},\delta)$ can serve our purpose. 

When $t\neq 0$, by \eqref{monoproj} we can choose $\{w_1,\dots, w_n\}$ as the local coordinate system on the (local) fibre $\tilde{\mU}_t(\tilde{x},\delta)=\tilde{\mU}(\tilde{x},\delta)\cap \tilde{\mX}_t$: 
\begin{equation}\label{w0tw'}
w_0=w_0(t,w_1,\dots, w_n)=w_0(t,w')=\frac{t}{\prod_{i=1}^{N_{\tilde{x}}}w_i^{a_i}},
\end{equation}
so that  for $t\neq 0$ we have:
\[
\textstyle{
\tilde{\mU}_t(\tilde{x},\delta)=\left\{\left(\frac{t}{\prod_{i=1}^{N_{\tilde{x}}}w_i^{a_i}},w'\right)\in\!\mathbb{C}^{n+1};\! |w_j|\le \delta, j=1, \dots, n,\! \mbox{ and } \!\left|\frac{t}{\prod_{i=1}^{N_{\tilde{x}}}w_i^{a_i}}\right|\le \delta\right\}.
}
\]
So $\tilde{\mU}_t(\tilde{x}, \delta)$ is biholomorphic to the following region in the $w'$-space via projection:
\[
\mathcal{V}'_t(\delta)=\left\{w'=(w_1,\dots, w_n)\in\mathbb{C}^n;\quad |w_j|\le \delta, j=1,\dots, n, \; \prod_{i=1}^{N_{\tilde{x}}}|w_i|^{a_i}\ge |t|/\delta\right\}.
\]
Note that $\{\mathcal{V}'_t(\delta)\}$ is an increasing sequence of sets on the $w'$-space with respect to the variable $t$. The limit is:
\[
\lim_{t\rightarrow0}\mathcal{V}'_t(\delta)=\{w'=(w_1,\dots, w_n)\in\mathbb{C}^n; |w_j|\le \delta, j=1,\dots, n\}=:\mathcal{V}_0'(\delta).
\]
The point here is that $w'=\{w_1,\dots, w_n\}$ serves as local coordinate system on $\mU_t(\tilde{x},\delta)$ for $t\neq 0$ and also on $\mU(\tilde{x},\delta)\cap \mX'_0$, so that we can use it to check that $\tilde{\mU}(\tilde{x},\delta)$ satisfies the local volume convergence condition as follows. To be clear, we will translate the formula \eqref{logres}-\eqref{adjunction} into analytic forms. We start with the equation \eqref{pullvol} that gives:
\begin{equation}\label{volume1}
\mu^* (v^{1/m})=g(w)\prod_{i=1}^{N_{\tilde{x}}} w_i^{a_i-b_i}\; (dw_0\wedge dw'\otimes \partial_t)\quad (\Longleftrightarrow \eqref{pullvol}),
\end{equation}
where $g(w)$ is a nowhere vanishing holomorphic function of $w$, and $dw'=dw_1\wedge\dots\wedge dw_n$.  
\begin{remark}\label{fracpower}
Precisely speaking, this formula should be interpreted as choosing a branch of the $1/m$-roots of following formula:
\[
\mu^*v=\tilde{g}(w)\prod_{i=1}^{N_{\tilde{x}}}w_i^{ma_i-m b_i} (dw_0\wedge dw')^{\otimes m}\otimes (\partial_t)^{\otimes m}.
\]
where $v$ is a generator of $\mathcal{O}_{\mX}(-mK_{\mX/\mathbb{C}})$ and $m b_i$ are integers. In the following, we will always implicitly assume that we have made this choice when we deal with fractional powers.
\end{remark}
From \eqref{volume1} and \eqref{w0tw'} we get
\begin{equation}
g(w)\frac{dw_0\wedge dw'\otimes\partial_t}{w_0}=\mu^*\left(\frac{v^{1/m}}{t}\right)\prod_{i=1}^{N_{\tilde{x}}}w_i^{b_i}. \quad (\Longleftrightarrow \eqref{logres}).
\end{equation}
%\begin{equation}
%K_{\tilde{\mX}/\mathbb{C}}+\mX_0'=\mu^*(K_{\mX/\mathbb{C}}+\mX_0)-\sum_{i=1}^K b_i E_i.
%\end{equation}
Taking residues, we get the analytic formula corresponding to equation \eqref{adjunction}:
\begin{equation}\label{analyticadjunction}
(\mu|_{\mX_0'})^* (v_0^{1/m})=g(0,w')\prod_{i=1}^{N_{\tilde{x}}} w_i^{-b_i} dw'. \quad (\Longleftrightarrow \eqref{adjunction})
\end{equation} 
When $t\neq 0$, we can compute the integral on $\tilde{\mU}_t:=\tilde{\mU}_t(\tilde{x},\delta) \subset \tilde{\mX}_t$ using local coordinates $\{w_1,\dots, w_n\}$ on it. 
%\begin{equation}\label{w0tw'}
%w_0=w_0(t,w_1,\dots, w_n)=w_0(t,w')=\frac{t}{\prod_{i=1}^{N_{\tilde{x}}}w_i^{a_i}}.
%\end{equation}
The local volume form in \eqref{volume1} restricted on $\tilde{\mU}_t$ becomes
\begin{eqnarray*}
\prod_{i=1}^{N_{\tilde{x}}}w_i^{a_i-b_i}dw\otimes\partial_t|_{\tilde{\mU}_t}&=&\prod_{i=1}^{N_{\tilde{x}}}w_i^{a_i-b_i}dw_0\wedge dw_1\wedge\cdots \wedge dw_n\otimes\partial_t|_{\tilde{\mU}_t}\\
&=&\prod_{i=1}^{N_{\tilde{x}}}w_i^{a_i-b_i}\frac{dt}{\prod_{i=1}^{N_{\tilde{x}}}w_i^{a_i}}\wedge dw_1\wedge\cdots\wedge dw_n\otimes\partial_t|_{\tilde{\mU}_t}\\
&=&\bigwedge_{i=1}^{N_{\tilde{x}}}w_i^{-b_i}dw_i\wedge\bigwedge_{j=N_{\tilde{x}}+1}^{n} dw_j.
\end{eqnarray*}
So by \eqref{volume1}, we have that: 
\[
\mu^*(v\wedge\bar{v})^{1/m}=|g(w_0(t,w'),w')|^2\bigwedge_{i=1}^{N_{\tilde{x}}}|w_i|^{-2b_i}dw_i\wedge d\bar{w}_i\wedge\bigwedge_{j=N_{\tilde{x}}+1}^n dw_j\wedge d\bar{w}_j.
\]
By \eqref{w0tw'} $\lim_{t\rightarrow 0}w_0(t,w')=0$. So we see that for any $w'\in\mathcal{V}'_t(\delta)$ for some $t$, we have the point-wise convergence:
\begin{eqnarray*}
\lim_{t\rightarrow 0}\mu^*(v\wedge\bar{v})^{1/m}&=&|g(0,w')|^2\bigwedge_{i=1}^{N_{\tilde{x}}}|w_i|^{-2b_i}dw_i\wedge d\bar{w}_i\wedge\bigwedge_{j=N_{\tilde{x}}+1}^{n}dw_j\wedge d\bar{w}_j\\
&=&\mu|_{\mX_0'}^* (v_0\wedge\bar{v}_0)^{1/m}.
\end{eqnarray*}
By \eqref{analyticadjunction}, the second equality is just the analytic form of formula \eqref{adjunction}. 
Now we can check the convergence in \eqref{localconvprop}:
\begin{eqnarray}\label{case1limit}
&&\lim_{t\rightarrow 0}\int_{\tilde{\mathcal{U}}_t(\tilde{x},\delta)}\mu^*(v\wedge\bar{v})^{1/m}=\lim_{t\rightarrow 0}\int_{\mathcal{V}'_t(\delta)} \frac{|g(w_0(t,w'),w')|^2}{\prod_{i=1}^{N_{\tilde{x}}}|w_i|^{2b_i}}dw'\wedge d\overline{w'}\\
&=&\lim_{t\rightarrow 0}\int_{\mathcal{V}'_0(\delta)}\chi(t,\delta)\frac{|g(w_0(t,w'),w')|^2}{\prod_{i=1}^{N_{\tilde{x}}}|w_i|^{2b_i}}dw'\wedge d\overline{w'}\nonumber\\
&=&\int_{\mV'_0(\delta)}\lim_{t\rightarrow 0}\left( \chi(t,\delta)\frac{|g(w_0(t,w'),w')|^2}{\prod_{i=1}^{N_{\tilde{x}}}|w_i|^{2b_i}}\right)dw'\wedge d\overline{w'}\nonumber\\
&=&\int_{\mV'_0(\delta)}\frac{|g(0,w')|^2}{\prod_{i=1}^{N_{\tilde{x}}}|w_i|^{2b_i}}dw'\wedge d\overline{w'}=\int_{\mX'_0\cap\tilde{\mU}(\tilde{x},\delta)}\mu|_{\mX'_0}^*(v_0\wedge\bar{v}_0)^{1/m}.\nonumber
%=\int_{\mU\cap\mX_0}(v_0\wedge\bar{v}_0)^{1/m}.
\end{eqnarray}
Here $\chi(t,\delta)$ is the characteristic function of the inclusion $\mV'_{t}(\delta)\subset\mV'_0(\delta)$. 
%The 3rd identity follows from Levi's monotone convergence theorem (or dominant convergence theorem). 
To see that the 3rd identity holds, we first note that because $g(w_0,w')$ is a holomorphic function on $\tilde{\mU}(\tilde{x},\delta)$, there exists an upper bound $|g(w_0, w')|\le M$. So we can estimate:
\[
\left|\chi(t,\delta)\frac{|g(w_0(t,w'),w')|^2}{\prod_{i=1}^{N_{\tilde{x}}}|w_i|^{2b_i}}\right|\le \frac{M}{\prod_{i=1}^{N_{\tilde{x}}}|w_i|^{2(b_i+\epsilon)}},
\]
where $\epsilon$ is small such that $b_i+\epsilon<1$. So the right-hand-side is integrable and, by dominated convergence theorem, we get the 3rd identity.
%\begin{remark}\label{remnbhd}
%The above calculation in \eqref{case1limit} essentially consists of two steps.  First we reduce the local volume integrals on $\mU_{t}(\tilde{x},\delta)$ to the integral on the image of the projection of $\mU_{t}(\tilde{x},\delta)$ to $\mX_0'$. Secondly, as $t\rightarrow 0$, these integrals converge to the volume integral on $\mX'_0\cap \mU(\tilde{x},\delta)$ because the domain of integrals converge and the positive integrands also converge under domination. So we actually see that any small simply-connected neighborhood $\tilde{\mU}(\tilde{x},\delta)$ of $\tilde{x}$ satisfies the local convergence property in \eqref{localconvprop}. For example, we may well choose the balls:
%\[
%\tilde{\mU}(\tilde{x}, \delta)=\left\{\sum_{i=0}^{n}|w_i|^2\le \delta^2\right\}.
%\]
%Here for clarity of calculation, we choose the small box around $\tilde{x}$ in \eqref{boxnbhd}. The same remark applies to the following case of calculations.
%\end{remark}
\item [Case 2:]
$\tilde{x} \not\in\mX_0'$. Without loss of generality, we assume 
\[
\tilde{x}\in \bigcap_{i=1}^{N_{\tilde{x}}}E_i.
\]
Similarly as before, we can choose the coordinate chart $w'=\{w_1,\dots, w_{n+1}\}$ such that $E_i=\{w_i=0\}$ for $1\le i\le N_{\tilde{x}}$ and the map $\tilde{\pi}$ is defined by the formula
\[
t=\prod_{i=1}^{N_{\tilde{x}}} w_i^{a_i}=\prod_{i=1}^{N_{\tilde{x}}} y_i,
\]
where for convenience we introduce $y_i=w_i^{a_i}$ with $\mathbb{Z}\ni a_i\ge 1$. So on $\tilde{\mX}_t$, we have 
\begin{equation}\label{y1}
y_1=w_1^{a_1}=\frac{t}{\prod_{i=2}^{N_{\tilde{x}}}w_i^{a_i}}=\frac{t}{\prod_{i=2}^{N_{\tilde{x}}}y_i}.
\end{equation}
Similarly as before, we consider the region:
\[
\tilde{\mU}(\tilde{x},\delta)=\{(w_1,\dots, w_{n+1}); |w_j|\le\delta, j=1,\dots, n+1\}; 
\]
So when $t\neq 0$, we have
\[
\tilde{\mU}_t(\tilde{x},\delta)=\{w'=(w_1,\dots, w_{n+1})\in \mathbb{C}^{n+1}; |w_j|\le \delta, j=1,\dots,n+1, \prod_{i=1}^{N_{\tilde{x}}}w_i^{a_i}=t\}.
\]
Denote $w''=\{ w_2,\dots, w_{n+1}\}$. Under the projection to the $w''$-space. $\tilde{\mU}_t(\tilde{x},\delta)$ is an unbranched $a_1$-fold covering over the following region:
\begin{equation}\label{tmV}
\tilde{\mV}_t''(\delta)=\left\{ (w_2,\dots, w_{n+1})\in\mathbb{C}^n; |w_2|\le \delta, j=2,\dots, n+1, \frac{|t|}{\prod_{i=2}^{N_{\tilde{x}}}|w_i|^{a_i}}\le \delta^{a_1} \right\}.
%&=&\left\{\right\}.
\end{equation}
%For later use, note that $\tilde{\mV}''_t(\delta)$ is a covering to the following region $\mV''_{t}(\delta)$:
%\begin{eqnarray*}
%\tilde{\mV}''_{t}(\delta)&=&\left\{(y_2,\dots, y_{N_{\tilde{x}}}, w_{N_{\tilde{x}}+1}, w_{n+1}); |y_j|\le \delta^{a_j}, j=2,\dots, N_{\tilde{x}}, \prod_{i=2}^{N_{\tilde{x}}}|y_i|\ge |t|/\delta, |w_j|\le \delta, 
%j=N_{\tilde{x}},\dots, n+1 \right\},
%\end{eqnarray*}
%\begin{eqnarray*}
%a_0 w_0^{a_0-1}dw\otimes\partial_t|_{\tilde{\mX}_t}&=&d(w_0^{a_0})\wedge dw_1\cdots \wedge dw_{n}\otimes\partial_t|_{\tilde{\mX}_t}\\
%&=&\frac{1}{\prod_{i=1}^{N_{\tilde{x}}}w_i^{a_i}}dt\wedge dw_1\wedge\cdots\wedge dw_{n}\otimes\partial_t|_{\tilde{\mX}_t}\\
%&=&\frac{1}{\prod_{i=1}^{N_{\tilde{x}}}w_i^{a_i}}dw_1\wedge\cdots\wedge dw_{n}|_{\tilde{\mX}_t}
%\end{eqnarray*}
Next we compute the integrands on $\tilde{\mU}_t(\tilde{x},\delta)$. Denote $dw'=dw_1\wedge\dots\wedge dw_{n+1}$. Then by \eqref{pullvol}, 
\begin{equation}\label{volume2}
\mu^* (v^{1/m})=g(w)\prod_{i=1}^{N_{\tilde{x}}} w_i^{a_i-b_i}\; (dw'\otimes \partial_t),
\end{equation}
%One should compare this formula with \eqref{volume1} which has extra $dw_0$ factor.
For convenience,  we rewrite the corresponding factors in \eqref{volume2} using the variable $y_i=w_i^{a_i}$ for $1\le i\le N_{\tilde{x}}$:
\begin{equation}\label{wivol}
w_i^{a_i-b_i}dw_i=\frac{1}{a_i} w_i^{1-b_i}d (w_i^{a_i})=\frac{1}{a_i} y_i^{(1-b_i)/a_i}d y_i=\frac{1}{a_i} y_i^{\beta_i} d y_i.
\end{equation}
Here we denote $\beta_i=(1-b_i)/a_i$. The important inequalities for us are $\beta_i>0$ because $b_i<1$ in \eqref{logres} (by inversion of adjunction) and $\mathbb{Z}\ni a_i\ge 1$, for $i=1,\dots, N_{\tilde{x}}$. 
%Note that since we only need to prove the following estimate locally, the $y_i$ can be any local branch of $w^{(1-b_i)/a_i}$. 

Using relations \eqref{y1} and \eqref{wivol}, we can find the volume form on $\tilde{\mU}_t:=\tilde{\mU}_t(\tilde{x},\delta)\subset\tilde{\mX}_t$ when $t\neq 0$:
\begin{eqnarray*}
&&\left.\prod_{i=1}^{N_{\tilde{x}}} w_i^{a_i-b_i}\; dw'\otimes\partial_t\right|_{\tilde{\mU}_t}=\left.w_1^{a_1-b_1} dw_1\wedge\prod_{i=2}^{N_{\tilde{x}}} w_i^{a_i-b_i}dw_i\wedge\bigwedge_{j=N_{\tilde{x}}+1}^{n+1}dw_j \otimes\partial_t\right|_{\tilde{\mU}_t}\\
&=&\left.\frac{1}{\prod_{i=1}^{N_{\tilde{x}}}a_i}\left(\frac{t}{\prod_{i=2}^{N_{\tilde{x}}}y_i}\right)^{\beta_1}\frac{dt}{\prod_{i=2}^{N_{\tilde{x}}}y_i}\wedge\bigwedge_{i=2}^{N_{\tilde{x}}} y_i^{\beta_i}d y_i\wedge\bigwedge_{j=N_{\tilde{x}}+1}^{n+1}dw_j\otimes\partial_t\right|_{\tilde{\mU}_t}\\
&=&\frac{t^{\beta_1}}{A}  \bigwedge_{i=2}^{N_{\tilde{x}}} y_i^{\beta_i-\beta_1-1} dy_i\wedge\bigwedge_{j=N_{\tilde{x}}+1}^{n+1} dw_j.
\end{eqnarray*}
where $A=\prod_{i=1}^{N_{\tilde{x}}}a_i$ is a positive integer. So 
\begin{equation}\label{volumecase2}
\mu^*(v\wedge\bar{v})^{1/m}= \frac{|g(w)|^2}{A^2} |t|^{2\beta_1} \bigwedge_{i=2}^{N_{\tilde{x}}} (|y_i|^{2(\beta_i-\beta_1)-2}dy_i\wedge d\bar{y}_i)\wedge\bigwedge_{j=N_{\tilde{x}}+1}^{n+1} dw_j\wedge d\bar{w}_j.
\end{equation}
Now note that, since $y_i=w_i^{a_i}\neq 0$ for $t\neq 0$, $\tilde{\mV}''_t(\delta)$ in \eqref{tmV} is an unbranched covering over the following region $\mV''_{t}(\delta)$:
\begin{eqnarray*}
\mV''_{t}(\delta)&=&\left\{(y_2,\dots, y_{N_{\tilde{x}}}, w_{N_{\tilde{x}}+1},\dots, w_{n+1}); |y_j|\le \delta^{a_j}, j=2,\dots, N_{\tilde{x}}, \right.\\
&&\hskip 2cm\left.\prod_{i=2}^{N_{\tilde{x}}}|y_i|\ge |t|/\delta^{a_1}, \mbox{ and } |w_j|\le \delta, 
j=N_{\tilde{x}}+1,\dots, n+1 \right\}\\
&=&\underline{\mathcal{V}}''_t(\delta)\times (S^1)^{N_{\tilde{x}}-1}\times \prod_{j=N_{\tilde{x}}+1}^{n+1}\{|w_j|\le \delta\}
\end{eqnarray*}
%$\tilde{\mV}''_t(\delta)$ is a covering to the following region $\mV''_{t}(\delta)$:
%\[
%\tilde{\mathcal{U}}_{\delta,t}=\{|y_j|\le \delta, j=2,\dots, n; \quad \prod_{i=2}^{N_{\tilde{x}}}|y_i|\ge |t|/\delta \},
%\]
%where $1>\delta>0$ is a fixed small positive number.
where we define:
\begin{eqnarray*}
\underline{\mathcal{V}}''_t(\delta)&=&\{(x_2,\dots, x_{N_{\tilde{x}}})\in\mathbb{R}^{n-1}; x_i\le a_i\log \delta, i=2,\dots, N_{\tilde{x}},\\
&&\hskip 4cm \mbox{ and } \sum_{j=2}^{N_{\tilde{x}}}x_j\ge \log|t|-a_1\log\delta \}.
\end{eqnarray*}
This is easily checked by changing into logarithmic polar coordinates: $y_i=e^{x_i+i\theta_i}$. Moreover, we compute that $
|y_i|^{2(\beta_i-\beta_1)-2}dy_i\wedge d\bar{y}_i=e^{2(\beta_i-\beta_1)x_i} \sqrt{-1} dx_i\wedge d\theta_i$. 
So by \eqref{volumecase2}, to estimate the integral $\int_{\tilde{\mU}_{t}(\tilde{x},\delta)}\mu^*(v\wedge\bar{v})^{1/m}$, 
we just need to estimate:
%\[
%|t|^{2\beta_1}\int_{\stackrel{x_i\le \log\delta}{\sum_{i}x_i\ge \log |t|-\log\delta}} \prod_{i=2}^{N_{\tilde{x}}}e^{2(\beta_i-\beta_1)x_i}\bigwedge_{i=2}^{N_{\tilde{x}}} dx_i.
%\]
\begin{equation}\label{realintegral}
|t|^{2\beta_1}\int_{\underline{\mV}''_t(\delta)} \prod_{i=2}^{N_{\tilde{x}}}e^{2(\beta_i-\beta_1)x_i}\bigwedge_{i=2}^{N_{\tilde{x}}} dx_i.
\end{equation}
Now note that we can certainly assume $0<\delta\le 1$. So the integral region $\underline{\mV}''_t(\delta)$ is contained in $(\mathbb{R}_{<0})^{N_{\tilde{x}}-1}$ and hence in a regular simplex whose edge has length $a_1\log\delta-\log |t|$. Up to now, the index $1$ hasn't played a special role. Now without loss of generality, we can assume that $\beta_i\ge \beta_1$, so that $(\beta_i-\beta_1)x_i\le 0$ in the region of integration. Combining these facts, we see that the integral in \eqref{realintegral} is bounded by
\[
|t|^{2\beta_1}(a_1\log\delta-\log |t|)^{N_{\tilde{x}}-1}/(N_{\tilde{x}}-1)!,
\]
which goes to $0$ as $t\rightarrow 0$ because $\beta_1>0$. So we have managed to prove:
\[
\lim_{t\rightarrow 0}\int_{\tilde{\mU}_{t}(\tilde{x},\delta)}\mu^*(v\wedge\bar{v})^{1/m}=0.
\]
%\[
%\mu^*(v\wedge\bar{v})^{1/m}=g(w_0(t,w'),w')|w_0|^{2(1-b_1)}\bigwedge_{i=1}^{N_{\tilde{x}}} |w_i|^{-2b_i}dw_i\wedge d\bar{w}_i.
%\]
\begin{remark}
For comparison, note that if we calculate the case 1 using the variable $y_i=w_i^{a_i}$ with $y_0=w_0$ and $a_0=1$, then we can let $\beta_0=0$ and get:
\[
\mu^*(v\wedge \bar{v})^{1/m}=\frac{|g(w)|^2}{\prod_{i=1}^{N_{\tilde{x}}}a_i}\bigwedge_{i=1}^{N_{\tilde{x}}} (|y_i|^{2\beta_i-2}dy_i\wedge d\bar{y}_i)\wedge\bigwedge_{j=N_{\tilde{x}}+1}^{n}dw_j\wedge d\bar{w}_j.
\]
This does not have extra $t_1$ factor (compare \eqref{volumecase2}).
Then one can carry out the calculation of the limit in $y_i$-coordinates, which is equivalent to that in \eqref{case1limit}.
\end{remark}
\end{enumerate}
\begin{remark}\label{logmodcont}
In the log setting, the change we need to make is to apply the inversion of adjunction in \cite[Theorem 5.50]{KM98} for the pair $(\mX_0, \alpha \mD_0)$ to get that the pair $(\mX, \mX_0+\alpha \mD)$ is plt. We have the corresponding analytic formulas 
for the log volume form (see \eqref{logvolumeform}):
\[
dV\left((\mX_t,(1-\beta)\mD_t); h_{\Omega}e^{-\Phi}|_{\mX_t}\right)=e^{-r(\beta)\Phi}\frac{|v_t^*|^{2/m}(v_t\wedge \bar{v}_t)^{1/m}}{\left.|\mathcal{S}|_{h_\Omega}^{2(1-\beta)}\right|_{\mX_t}}.
\]
and local convergence properties remains valid essentially because $(\mX, \mX_0+(1-\beta)\mD)$ is plt.
\end{remark}

%\end{proof}
%\begin{remark}
%Professor Robert Berman pointed out to me that the continuity of $\rm II$ was not known in previous work. So Lemma \ref{IIcont} also answers his question. 
%\end{remark}

\begin{remark}\label{DKG}
One should compare Lemma \ref{IIcont} with \cite[Proposition 2.1]{Ti90} and \cite[Main Theorem]{PS00} (see also \cite{DeKo}), where similar local problems for holomorphic functions on the product  space were considered. Here the topology changes near the central fibre and we need to pull back all the calculations to a log-resolution. 

On the other hand, one referee has pointed out that the general Lemma stated in Lemma \ref{contlem} can be viewed as a strengthening of one result of M.~Gross \cite[Appendix B]{RoZh} which in our set-up says that the integrals on the left hand side of \eqref{globalconvergence1} is uniformly bounded. Our proof of the Lemma is quite different from that of Gross. While he used the deep theory of mixed Hodge structures, our proof uses elementary local computations that however rely on one important result in birational geometry: inversion of adjunction. Note that in the same setup, ``inversion of adjunction"  was first used by Berman \cite{Berm2} to show that ${\rm II}(t)$ has no logarithmic pole at $t=0$. In his paper (see \cite[Proof of Theorem 2.8]{Berm2}), Berman also speculated that ${\rm II}(t)$ should be continuous at $t=0$. So Lemma \ref{IIcont} is a confirmation of his speculation.

Because computations in the proof of Lemma \ref{IIcont} are carried out locally, the proof actually works without the global relative Fano condition. In particular, one should be able to show the following result:

%\begin{prop}
\textit{
Let $\pi: \mX\rightarrow B_1(0)$ be a family of projective varieties over the unit-disc such that the general fiber is smooth and the central fiber has log terminal singularities. Let $(\mathcal{L}, h)$ be a holomorphic line bundle over $\mX$ which is semi-positive on $\mX$ and $h$ is a continuous metric on $\mL$ with positive curvature current: $-\sddbar\log h\ge 0$. Assume further that $K_{\mX/B_1(0)}+\mL$ is $\pi$-free (\cite[Definition 3.22]{KM98}). Then the relative Bergman kernel metric induced by $h$ on 
$K_{\mX/B_1(0)}+\mL$ is a continuous metric.
}
%\end{prop}
\end{remark}

\section{Examples}

\begin{enumerate}
\item (Smooth examples)
Mukai-Umemura Fano 3-fold $X_0$ \cite{MU} is a smooth compactification of $SL/\Gamma$ where $\Gamma$ is the binary icosahedral group. $X_0$ has a large symmetry group $SL(2,\mathbb{C})$. Using the method of Mukai, one can study smooth deformations of $X_0$. It was Tian \cite{Ti97} who first used the generic deformation $X_1$ of $X_0$ to construct a special degeneration of $X_1$ to $X_0$. Tian's K-stability then proves that $X_1$ does not have a K\"{a}hler-Einstein metric although there is no holomorphic vector field on $X_1$. Donaldson \cite{Do08} proved the existence of K\"{a}hler-Einstein metric on $X_0$ using Tian's $\alpha$-invariant. So $X_1$ is a K-semistable but not K-polystable Fano manifold. Because $X_0$ is smooth, this was already pointed out by Chen \cite{Chen}. There are other smooth examples of this kind in the recent work of S\"{u}\ss \ \cite{Sus13}.

\item (A singular logarithmic example revisited)
Here we revisit a class of examples from our previous work in \cite[Section 3.3]{LiSu}. Assume that $X$ is a Fano manifold and $D$ is a smooth divisor such that $D\sim -\lambda K_X$ with $0<\lambda\in \mathbb{Q}$. By the adjunction formula $-K_{D}=(1-\lambda) (-K_X)|_{D}$. From now on,  we assume $\lambda<1$. So $-K_D$ is ample and $D$ is again a Fano manifold. There is a construction of a special degeneration of the pair $(X,\alpha D)$ (for any $\alpha\in [0,1)$) via the deformation to normal cone. For this first let $\mathcal{Y}=Bl_{D\times\{0\}}(X\times\mathbb{C})$ be the blow-up of the product complex manifold $X\times\mathbb{C}$ along the smooth complex submanifold $D\times \{0\}$. The central fibre $\mathcal{Y}_0$ is the union of two components $X\cup E$ where the $X$ component is the strict transform of $X\times\{0\}$ which is unchanged because $D\times\{0\}$ is of codimension one in $X\times\{0\}$. $E$ denotes the exceptional divisor, which in this case is nothing but $\mathbb{P}(N_D\oplus\mathbb{C})$ where $N_D$ is the normal bundle of $D\subset X$. We also denote by $\mathcal{D}$ the strict transform of $D\times \mathbb{C}$ in $\mathcal{Y}$. It's easy to see that $\mathcal{D}\cong D\times\mathbb{C}$.

We have a line bundle $\mathcal{L}_c=\pi_1^*(-K_X)-c E$ on $\mathcal{Y}$. It's easy to see that $\mathcal{L}_c$ is relatively ample on $\mathcal{Y}$ (over $\mathbb{C}$) if and only if $c\in (0,\lambda^{-1})$ (See \cite[Lemma 3.13]{LiSu}). Moreover, $\mathcal{L}_{\lambda^{-1}}$ is semi-ample over $\mathcal{Y}$ and the linear system $|-m\mathcal{L}_{\lambda^{-1}}|$ for sufficiently large $m$ gives a map $\tau: \mathcal{Y}\rightarrow \mX$ by contracting the component $X$ in the central fibre, and we have $\mathcal{L}_{\lambda^{-1}}^{\otimes m}\sim_{\mathbb{C}}\tau^*(-K_{\mX/\mathbb{C}}^{\otimes m})$. In this way $(\mX, \alpha\mD, -K_{\mX})$ becomes a special degeneration of the polarized pair $(X,\alpha D,-K_X)$. The central fibre $\mX_0$ is obtained from $E=\mathbb{P}(N_D\oplus\mathbb{C})$ by contracting the infinity section $D_\infty$ of the $\mathbb{P}^1$-bundle and hence has an isolated singularity. On the other hand $\mD_0$ is the zero section $D_0$ of $E$ which does not change under $\tau$. Because $\tau|_E: E\rightarrow \mX_0$ is a resolution of singularity, we can write
%\[
$K_{E}=(\tau|_E)^*K_{\mX_0}+a(\mX_0, D_\infty) D_\infty$.
%\]
Using adjunction formula, we can get the discrepancy $a(D_\infty, \mX_0)=(1-2\lambda)/\lambda>-1$ when $0<\lambda<1$. Note that $\mX_0$ is smooth along $\mD_0$. So the pair $(\mX_0, \alpha \mD_0)$ is klt if and only if $\alpha\in [0,1)$.

\begin{lem}\label{rotationsym}
Assume that $D$ admits a smooth K\"{a}hler-Einstein metric $\omega^{D}_{KE}$. Then there exists a rotationally symmetric conical K\"{a}hler-Einstein metric on the pair $(\mX_0, \left(1-\beta\right)\mD_0)$ with $\beta=\frac{\lambda^{-1}-1}{n}$.
\end{lem}
Assuming this lemma, by our Theorem \ref{bdDing} and Theorem \ref{main}, we get a corollary:
\begin{cor}
Under the same assumption as in the above Lemma with $\beta=\frac{\lambda^{-1}-1}{n}$, the log-Ding-energy of $(X, (1-\beta) D)$ is bounded from below. Hence $(X, (1-\beta) D, -K_X)$ is log-K-semistable but not log-K-polystable .
\end{cor}
In \cite{LiSu}, we considered a special case when $X=\mathbb{P}^2$ and $D=\{Z_0^2+Z_1^2+Z_2^2=0\}$ so that $\lambda=\frac{2}{3}$ and $\beta=\frac{1}{4}$. In this case, $(\mX_0, \mD_0)\cong (\mathbb{P}^2(1,1,4), \{Z_2=0\})$ and the conical K\"{a}hler-Einstein metric is nothing but the standard orbifold metric on $\mathbb{P}^2(1,1,4)$ coming from the branched covering $\mathbb{P}^2\rightarrow \mathbb{P}^2(1,1,4)$. Here we observe that this is just a special example of the above general set-up.
\begin{proof}[Proof of Lemma \ref{rotationsym}]
The construction is similar to the construction of rotationally symmetric K\"{a}hler-Ricci solitons in the author's thesis \cite{Li12} which was a generalization of earlier constructions by Koiso, Cao, and also Feldman-Ilmann-Knopf. First, for later convenience, we rescale the K\"{a}hler-Einstein metric $\omega^{D}_{KE}$ on $D$ to be contained in the class $2\pi c_1(-K_X)|_D=2\pi (1-\lambda)^{-1} c_1(-K_D)$. Then we can choose a Hermitian metric $h$ on $N_D\rightarrow D$ such that
$-\sddbar\log h=\lambda \omega_{KE}^D$ because $N_D=-\lambda K_X|_D$. We will view $h$ as a positive function denoted by $r$ on the total space of the line bundle $\pi: N_D\rightarrow D$. Then we construct the Calabi ansatz $
\omega=\pi^*\omega^{D}_{KE}+\sddbar P(r)$. To calculate it, we choose a local trivialization of $N_D$ so that $h=a(z)|\xi|^2$ where $z=\{z_1,\dots, z_{n-1}\}$ is a coordinate chart on $D$ and $\xi$ is the holomorphic coordinate along the fibre. 
By a straight forward calculation we get:
\begin{equation}\label{calans}
\omega=(1-\lambda P_r r)\omega^{D}_{KE}+(P_rr)_rr\frac{\nabla \xi\wedge \overline{\nabla \xi}}{|\xi|^2}=(1-\lambda P_s)\omega^{D}_{KE}+P_{ss}\frac{\nabla \xi\wedge \overline{\nabla\xi}}{|\xi|^2}.
\end{equation}
We have introduced $s=\log r\in (-\infty, +\infty)$ and denoted the horizontal cotangent differential by:
\[
\nabla \xi=d\xi+\xi a^{-1}\partial a=\xi\cdot  \partial\log h.
\]
From \eqref{calans}, we see that the necessary condition for $\omega$ to be positive definite is:
\begin{equation}\label{posdef}
P_s\in [0, \lambda^{-1}), \mbox{ and } P_{ss}>0.
\end{equation}
In particular, $P$ is a convex function and $P_s$ is increasing on $(-\infty, +\infty)$.
From \eqref{calans}, we can also calculate the volume form:
\begin{equation}\label{rotvol}
\omega^n=n(1-\lambda P_s)^{n-1}P_{ss}(\omega^{D}_{KE})^{n-1}\wedge\frac{d\xi\wedge d\bar{\xi}}{|\xi|^2}.
\end{equation}
Suppose that we prescribe the angle $\beta$ along the zero section $D_0$ of $N_D\rightarrow D$. Then we would like to solve the equation:
\begin{equation}\label{rotcKE}
Ric(\omega)=\mu_\beta\omega+(1-\beta)\{D_0\}, \quad \mbox{ on } N_D.
\end{equation}
By taking the cohomology class and restricting to $D_0$ we can determine 
\begin{equation}\label{mubeta}
\mu_\beta=1-\lambda(1-\beta)=1-\lambda+\lambda\beta.
\end{equation}
Under local trivialization, (by Lelong-Poincar\'{e} formula) the right-hand-side is equal to 
\[
\mu_\beta\omega+(1-\beta_1)\{D_0\}=\sddbar(\mu_\beta(-\lambda^{-1}\log a+P(s))+(1-\beta)\log|\xi|^2).
\]
Using this and \eqref{rotvol}, we can reduce the equation \eqref{rotcKE} to the following equation:
\begin{eqnarray*}
&&-\sddbar\left((n-1)\log(1-\lambda P_s)+\log P_{ss}\right)+Ric(\omega^{D}_{KE})+\sddbar\log|\xi|^2\\
&&\hskip 4cm=\sddbar(\mu_\beta(-\lambda^{-1}\log a+P(s))+(1-\beta)\log |\xi|^2).
\end{eqnarray*}
Now, according to our normalization, we have $\scriptstyle Ric(\omega^D_{KE})=(1-\lambda)\omega_{KE}^D=(1-\lambda)(-\lambda^{-1}\sddbar\log a)$.
Substituting this into above equation, we can reduce \eqref{rotcKE} to the following ordinary differential equation:
\begin{equation}
(n-1)\log(1-\lambda P_s)+\log P_{ss}=\beta s-\mu_\beta P+constant.
\end{equation}
Taking the derivative with respect to $s$, we get, using \eqref{mubeta}
\begin{equation}\label{rotode}
(n-1)\frac{-\lambda P_{ss}}{1-\lambda P_s}+\frac{P_{sss}}{P_{ss}}=\beta-\mu_\beta P_s=(1-\lambda^{-1})+\mu_\beta\lambda^{-1}(1-\lambda P_s).
\end{equation}
Introduce a new variable $\phi=P_s$. Since $\phi_s=P_{ss}>0$ by \eqref{posdef}, we can write $s=s(\phi)$ and define $F(\phi)=\phi_s(s(\phi))$ so that
$F'(\phi)=\phi_{ss}s_\phi=P_{sss}/P_{ss}$. So we reduce \eqref{rotode} to 
\begin{equation}
(n-1)\frac{-\lambda F(\phi)}{1-\lambda\phi}+F'(\phi)=-(\lambda^{-1}-1)+\mu_\beta(1-\lambda\phi).
\end{equation}
Now multiplying the integrating factor $(1-\lambda\phi)^{n-1}$ we can solve the equation:
\begin{equation}\label{relateFut}
(1-\lambda\phi)^{n-1}F(\phi)=\frac{\lambda^{-1}-1}{n\lambda}((1-\lambda\phi)^{n}-1)-\frac{\mu_\beta\lambda^{-1}}{(n+1)\lambda}((1-\lambda\phi)^{n+1}-1).
\end{equation}
Now we finally bring $\mX_0$ into picture. To get isolated singularity at infinity, we need $\lim_{s\rightarrow+\infty}(1-\lambda\phi(s))=0$. So we get $\lim_{s\rightarrow+\infty}\phi(s)=\lambda^{-1}$.
So by taking limits on both sides of above identity, we get:
\[
0=\frac{\lambda^{-1}-1}{n}-\frac{\mu_\beta\lambda^{-1}}{n+1}\Longrightarrow \beta=\frac{\lambda^{-1}-1}{n}.
\]
For this $\beta$ we can get:
\[
\phi_s=F(\phi)=\frac{\lambda^{-1}-1}{n\lambda}((1-\lambda\phi)-(1-\lambda\phi)^2)=\frac{\beta}{\lambda}((1-\lambda\phi)-(1-\lambda\phi)^2).
\]
So we can find the explicit metric and potential for the conical K\"{a}hler-Einstein metric:
\[
%1-\lambda \phi(s)=\frac{1}{1+C e^{\beta s}}\Longrightarrow 
P_s=\phi(s)=\frac{1}{\lambda}\frac{1}{1+Ce^{-\beta s}}\Longrightarrow P(s)=\frac{1}{\lambda\beta}\log(1+C^{-1}e^{\beta s}).
\]
The positive constant $C$ clearly represents the transformation of the conical K\"{a}hler-Einstein metric by the $\mathbb{C}^*$-action on $(\mX_0, \mD_0)$.
\end{proof}
\begin{remark}
The identity \eqref{relateFut} is closely related to the calculation of log-Futaki-invariant in the proof of \cite[Proposition 13]{LiSu}.
\end{remark}
\begin{remark}
If we assume that $X$ itself also admits a smooth K\"{a}hler-Einstein metric, then by the interpolation argument in \cite{LiSu}, $(X,(1-t)D)$ admits a conical K\"{a}hler-Einstein metric if and only if $t\in \left(\beta=\frac{\lambda^{-1}-1}{n}, 1\right]$. On the other hand, H.J. Hein first suggested to the author and S.Sun that one should always be able to glue Tian-Yau's Calabi-Yau metric to the above singular K\"{a}hler-Einstein metric and perturb the angle to get the conical K\"{a}hler-Einstein metric on $(X, (1-\gamma)D)$ for $\gamma$ slightly bigger than $\beta$.  We plan to study this gluing problem in future together with H.J.Hein and S.Sun.
\end{remark}

\end{enumerate}

\noindent
{\bf Acklowledgement:}
The author would like to thank Professor Gang Tian for constant encouragement.  The author would like to thank Dr. Song Sun for stimulating discussions during the joint work in \cite{LiSu}, and Professor Yanir Rubinstein for bringing the references of \cite{LH} and \cite{Rub} to his attention.  The author is especially grateful to Professor Robert Berman for kindly pointing out that the argument for the continuity of $\rm II$ was missing in the first version and for very helpful discussions. The author would also like to thank the referees for carefully reading the manuscript and constructive suggestions on the improvement of presentation. In particular, one referee pointed out the relation of Lemma \ref{contlem} to the result of M.~Gross.

\noindent
Mathematics Department, Stony Brook University, Stony Brook NY, 11794-3651, USA \\
Email: chi.li@stonybrook.edu

\end{document}